\documentclass{article}
\usepackage{arxiv}
\usepackage{amsmath}
\usepackage{amsthm}
\usepackage{graphicx}

\usepackage{algpseudocode}
\usepackage{boondox-cal}
\usepackage[T1]{fontenc}
\usepackage{lipsum}
\usepackage{epstopdf}
\usepackage{amsfonts}
\usepackage{graphicx}
\usepackage{cancel}
\usepackage[english]{babel}
\usepackage{booktabs}
\usepackage{amssymb}
\usepackage{subfigure}
\usepackage{color}

\ifpdf
  \DeclareGraphicsExtensions{.eps,.pdf,.png,.jpg}
\else
  \DeclareGraphicsExtensions{.eps}
\fi
\newcommand{\cD}{\mathcal{D}}
\newcommand{\cO}{\mathcal{O}}

\newcommand{\IR}{\mathbb{R}}

\newcommand{\IN}{\mathbb{N}}

\newcommand{\vr}{\mathbf{r}}

\newcommand{\vg}{\mathbf{g}}
\newcommand{\ve}{\mathbf{e}}

\newcommand{\vx}{\mathbf{x}}

\newcommand{\vn}{\mathbf{n}}

\newcommand{\vv}{\mathbf{v}}

\newcommand{\cm}{\mathcal{m}}
\newcommand{\cl}{\mathcal{l}}

\newcommand{\half}{\frac{1}{2}}
\newcommand{\fourth}{\frac{1}{4}}

\newcommand{\vlambda}{\boldsymbol{\lambda}}
\newcommand{\vnu}{\boldsymbol{\nu}}

\newtheorem{corollary}{Corollary}
\newtheorem{theorem}{Theorem}
\newtheorem{lemma}{Lemma}
\newtheorem{remark}{Remark}
\newtheorem{problem}{Problem}

\DeclareMathOperator{\spn}{span}

\begin{document}

\title{Fast Galerkin Method for Solving Helmholtz Boundary Integral Equations on Screens
\thanks{This work was partially funded by Fondecyt Regular 1171491 and grant Conicyt  21171479 Doctorado Nacional 2017.}}
% Short title for running heads:
\shorttitle{Fast Galerkin Method for Solving Helmholtz BIEs on Screens}

\author{%
{\sc
Carlos~Jerez-Hanckes\thanks{Facultad de Ingenier\'ia y Ciencias, Universidad Adolfo Ib\'a\~nez, Santiago, Chile. Email: carlos.jerez@uai.cl},
Jos\'e Pinto\thanks{School of Engineering, Pontificia Universidad Cat\'olica de Chile, Santiago, Chile, Email: jspinto@uc.cl}
}
}
% Short list of authors for running heads:

\maketitle

\begin{abstract}
% Body of abstract:
{We solve first-kind Fredholm boundary integral equations arising from Helmholtz and Laplace problems on bounded, smooth screens in three-dimensions with either Dirichlet or Neumann conditions. The proposed Galerkin-Bubnov method takes as discretization elements pushed-forward weighted azimuthal projections of standard spherical harmonics onto the canonical disk. We show that these bases allow for spectral convergence and provide fully discrete error analysis. Numerical experiments support our claims, with results comparable to Nystr\"om-type and $hp$-methods.
}
% Keywords:
{Boundary integral equations, spectral methods, wave scattering problems, screens problems, non-Lipschitz domains.}
\end{abstract}

%%%%%%%%%%%%%%%%%%%%%%%%%%%%%%%%%%%%%%%%%%%%%%%%%%%%%%%%%%%%
\section{Introduction}
%%%%%%%%%%%%%%%%%%%%%%%%%%%%%%%%%%%%%%%%%%%%%%%%%%%%%%%%%%%%
We study the solution of the Helmoltz and Laplace problems with Dirichlet or Neumann conditions posed on an open orientable bounded surface $\Gamma \subset \IR^3$. These can be summarized as follows:
\begin{problem}
\label{prob:edp}
Let $k \geq 0$, find $u$ defined in $\IR^3$, such that 
\begin{align*}
-\Delta u -k^2 u =0, \quad \text{on } \IR^3 \setminus\overline{\Gamma}, \\
u  = g_d \quad \text{or} \quad \partial_n u  = g_n \quad \text{in } \Gamma, \\
\text{condition at infinity}(k),
\end{align*}
where $\partial_n u$  is the normal derivative of $u$  with the normal defined via the parametrization used\footnote{A fixed orientation for the normal vector is required, and one can choose either of the two orientations that yield continuous normal vectors without loss of generality.}, $g_d$ and $g_n$ are suitable Dirichlet and Neumann data, respectively, and the condition at infinity reads 
\begin{align*}
\begin{cases}
\lim_{ \| \vx \| \rightarrow \infty } \| \vx \|  \left( 
 \partial_{\|\vx \|}u -iku
\right)=0, & k>0,\\
u(\vx) = O(\|\vx\|^{-1}),\text{ as } \|\vx	\| \rightarrow \infty, & k =0,
\end{cases}
\end{align*}
\end{problem} 
For a detailed discussion of these conditions see \cite[Chap.~8 and 9]{mclean2000strongly}. The Laplace case occurs when $k=0$.  

Problem \ref{prob:edp} and  associated boundary integral equations (BIEs) have been extensively studied \cite{Stephan1987, Stephan86,COSTABEL,HaDoung1990,HJU18,lintner_bruno_2015,RAN17,HJU20}. Indeed, BIEs rigorously recast the volume problem onto the screen while taking into account the corresponding condition at infinity. By meshing the open surface $\Gamma$, standard local low-order approximations can be built and solved by today's compression algorithms \cite{bebendorf2008a,yijun2009a} to accelerate computations. However, solving Problem \ref{prob:edp} remains far from trivial as standard boundary element implementations converge at the worst rate possible due to the solution's singular behavior near the screen surface. Thus, one has to resort to special meshing techniques \cite{Norber99} or to increase the polynomial degree of the underlying discretization in a suitable manner \cite{Norbert2002,Norbert2007} so as to recover better error convergence rates. This however is a tedious procedure when solving for multiple configurations as in uncertainty quantification or inverse problems.

In two-dimensional space, i.e.~when $\Gamma$ is an open arc, an alternative discretization of BIEs can be performed by means of weighted Chebyshev polynomials --spectral discretization--, which explicitly capture the edge singularity and allow for super-algebraic convergence whenever $g_d$ and $g_n$ are smooth functions \cite{saranen2013periodic,Jerez-Hanckes2012, ElPinto,JHP20}.  Opposingly, for screens in $\IR^3$, to the best of our knowledge, no such equivalent discretization exists. Moreover, though for closed surfaces an spectral BIE method was presented by Sloan \emph{et al.}~\cite{sloan2002}, the presence of extra singularities around the surface edges and the structure of the underlying fundamental solution renders impossible the extension of this method to the context of open surfaces directly. 

Our main contribution is the derivation and analysis of spectral Galerkin-Bubnov methods to tackle Laplace and Helmholtz BIEs on open surfaces in three-dimensional space. In contrast to the Nyst\"rom approach \cite{BRUNO2013250}, our spectral Galerkin method preserves all the theoretical properties of a Galerkin method and its implementation relies in suitable quadrature rules and change of variables without the need of special window functions, while maintaining the convergence rate for smooth geometries. On the other hand, while special $p$- and $hp$-discretizations (see  \cite{Norber99, Norbert2002,Norbert2007}) could also achieve super-algebraic convergence rates, and being more flexible for non-smooths geometries, they need a mesh which in practice could slowdown the implementation when solutions for different geometries are needed as in uncertainty quantification or inverse problem \cite{Kress95}.

Our work is structured as follows. In Section \ref{sec:SH} we define a family of finite-dimensional functional spaces by projecting the spherical harmonics into the unitary disk and consider their linear span. We then define the trial and test spaces for the Galerkin method by mapping the functions to the screen in consideration. 
From the definition of our trial space, we construct adequate auxiliary functional spaces which allow for the associated error analysis, and we describe how theses spaces are related to the standard Sobolev one. This fact, in conjunction with standard Galerkin properties for the variational versions of the BIEs (Section \ref{sec:BIOS}), lead to a semi-discrete a priori error bound in Theorem \ref{thrm:rateConvergence}.  Then, in Section \ref{sec:quaderror}, we detail the quadrature procedure used to approximate the action of the weakly- and hyper-singular boundary integral operators (BIOs) on the trial space, and provide a full error analysis which incorporate the quadrature error.  Finally, in Section \ref{sec:numres3d} we show some numerical examples that exhibit the super-algebraic convergence method.

%%%%%%%%%%%%%%%%%%%%%%%%%%%%%%%%%%%%%%%%%%%%%%%%%%%%%%%
\section{Mathematical Tools}
%%%%%%%%%%%%%%%%%%%%%%%%%%%%%%
This section introduces general definitions needed for the analysis of Problem \ref{prob:edp} and of the proposed Galerkin spectral method. 

Given $\vx, \vx' \in \mathbb{C}^d$, $d \in \IN$, the standard dot product is denoted by $\vx \cdot \vx' = \sum_{j=1}^d x_j \overline{x'}_j$, and the Euclidean norm satisfies $\| \vx \|^2 = \vx \cdot \vx$. For $d = 3$ and real vectors, we denote $\vx \times \vx'$ the cross-product --vectorial product. For $a,b \in \IR$ we will make use of the notation $a \lesssim  b$ whenever there is a positive number $C$ such that $a \leq C b$, typically $C$ is a constant not relevant for the underlying analysis. Also, we will denote by $L^2(A)$ the classical Lebesgue space of square-integrable functions over a measurable set $A \subset \IR^d$, $d = 2,3$.  

%%%%%%%%%%%%%%%%%%%%%%%%%%%%%%%%%%%%%%%%%%%%%%%%%%%%%%%%%%%%
\subsection{Geometry}
%%%%%%%%%%%%%%%%%%%%%%%%%%%%%%%%%%%%%%%%%%%%%%%%%%%%%%%%%%%%
Throughout, we denote $\Gamma \subset \IR^3$ a smooth orientable connected surface with boundary, also called smooth screen or simply screen. We assume that the screen is contained on a smooth orientable closed surface which will be denoted $\widetilde{\Gamma}$. The canonical disk and spheres are given, respectively, by
\begin{align*}
\mathbb{D}:= \{ \vx \in \IR^2 : \|\vx\| \leq 1 \}, \quad
\mathbb{S} := \{ \vx \in \IR^3 : \|\vx\| = 1 \}.
\end{align*} 
For any given screen, we assume that there is a parametrization $\vr : \mathbb{D}\rightarrow \Gamma $, being a smooth function such that in every argument and coordinate can be extended to an analytic function on a Bernstein ellipse in the complex plane of parameter\footnote{Given and interval $[a,b] \subset \IR$, the Bernstein ellipse of parameter $\rho$, corresponds to an ellipse on the complex plane with foci at $a,b$, semi-major axis $\frac{(b-a)(\rho+\rho^{-1})}{4}$, and semi-minor axis $\frac{(b-a)(\rho-\rho^{-1})}{4}$. } $\rho>1$. The class of functions portraying the regularity previously described are referred to as $\rho-$analytic. Furthermore, we impose that the gradients of $\vr$ (as a matrix) have full rank for any point, and the Jacobian, $J_\vr(\vx) = \| \partial_{x_1} \vr \times \partial_{x_2}
 \vr \|$ is such that $J_\vr(\vx)/\| \vx\|$ is bounded and nowhere null.
 
The direction of the unitary normal vector of $\Gamma$ is selected to be equal to  $\partial_{x_1} \vr \times \partial_{x_2}
 \vr$.  This imply that, rigorously speaking, a screen is characterized by the particular parametrization $\vr$ and no by the set of points that represent the physical domain $\Gamma$.

On $\mathbb{S}$, we will often consider spherical coordinates characterized by two angles $(\theta, \varphi) \in [0,2\pi]\times[0,\pi]$ . We will always assume that $\theta$ corresponds to the polar angle of a given point when projected to the $x_3=0$ plane, and $\varphi$ denotes the angle measured from the $x_3-$axis. Similarly, in $\mathbb{D}$ we use polar coordinates $(r,\theta)\in [0,1]\times[-\frac{\pi}{2},\frac{3\pi}{2}]$.

%%%%%%%%%%%%%%%%%%%%%%%%%%%%%%%%%%%%%%%%%%%%%%%%%%%%%%%%%%%%
\subsection{Classical Functional Spaces}
%%%%%%%%%%%%%%%%%%%%%%%%%%%%%%%%%%%%%%%%%%%%%%%%%%%%%%%%%%%%
Given a set $\cO \subset \IR^d$, $d \in \{2,3\}$, we denote by $\cD(\cO)$ the space of smooth functions with compact support in $\cO$ with topological dual $\cD(\cO)'$. If $\cO$ is open then $\cD(\cO)$ are smooth functions whose extension by zero is also smooth. If $\cO$ is compact then $\cD(\cO)$ are the set of infinity differentiable  functions. We will require Sobolev spaces defined on open and closed surfaces. For $\Gamma$ an open surface, the Sobolev spaces $H^s(\widetilde{\Gamma})$, $s \in \IR$, are defined by using local parametrizations and a partition of unity \cite[Chap.~2]{mclean2000strongly}. Spaces $H^s(\Gamma)$, and $\widetilde{H}^s(\Gamma)$ , with $s \in \IR$, are also standard and can be defined as 
\begin{align*}
H^s(\Gamma) &:= \{ u \in \cD'(\Gamma) : \exists \ U \in H^s(\widetilde{\Gamma}), \  u= U|_\Gamma\}, \\
\widetilde{H}^s(\Gamma) &:= \{ u \in H^s(\widetilde{\Gamma}):\ \mathrm{supp}(u)\subset \overline{\Gamma}\}, 
\end{align*}
where the support and restriction have to be understood in the context of distributions. One can identify the dual space of $H^{s}(\Gamma)$ with $\widetilde{H}^{-s}(\Gamma)$, the corresponding duality product is denoted $\langle \cdot, \cdot \rangle_\Gamma$. The duality product is an extension of the $L^2(\Gamma)$ inner product as we work under the identification, $H^0(\Gamma) = \widetilde{H}^0(\Gamma) = L^2(\Gamma)$. The following Lemma is useful to retain control over the norms of functions defined on an arbitrary screen. 
\begin{lemma}[Theorem 3.23 in \cite{mclean2000strongly}] 
\label{lemma:SobolevLocalization}
Let $s\in \IR$, if $u \in H^s(\Gamma)$, we can define an equivalent norm as 
\begin{align*}
\| u \circ \vr \|_{H^s(\mathbb{D})} \cong \|u\|_{H^s(\Gamma)},
\end{align*}
where the unspecified constant depends on $\Gamma$ only. The same result holds true when we change $H^s(\Gamma)$ for $\widetilde{H}^s({\Gamma})$ and $H^s(\mathbb{D})$ for $\widetilde{H}^s(\mathbb{D})$, accordingly.
\end{lemma}
\begin{remark}
Lemma \ref{lemma:SobolevLocalization} can be generalized to screens of restricted regularity by limiting the range of $s$ depending on the regularity (cf.~\cite[Theorem 3.2.3]{mclean2000strongly} for details). 
\end{remark}

Customarily, we extend the definitions of restrictions and normal derivatives over $\Gamma$ to linear bounded maps in appropriate Sobolev spaces.  In particular, following \cite[Chapter 2]{mclean2000strongly}, we define the Dirichlet traces as the maps $\gamma_d^\pm : H^s(\IR^3 \setminus  \overline{\Gamma}) \rightarrow H^{s-\half}(\Gamma)$ which extensions of the following operator:   
$$\gamma_d^\pm u(\vx) = \lim_{\epsilon \downarrow 0} u(\vx \pm \epsilon \widehat{\vn}(\vx)), \quad \vx\in \Gamma,$$
for $s > \half$, and where $\widehat{\vn}$ denotes the unitary normal vector whose direction depends on the parametrization $\vr$ of $\Gamma$. The Neumann trace $\gamma_n$ is defined as the extension of the normal derivative: 
$$\gamma_n^\pm u(\vx) = \lim_{\epsilon \downarrow 0} \nabla u(\vx \pm \epsilon \widehat{\vn}(\vx)) \cdot \widehat{\vn}(\vx), \quad \vx\in \Gamma, $$
 where the extension can be done from $H^s(\IR^3 \setminus  \overline{\Gamma}) \rightarrow H^{s-\half}(\Gamma)$ for $s > \frac{3}{2}$, or from $H^1(\IR^3 \setminus  \overline{\Gamma}) \rightarrow H^{-\half}(\Gamma)$ for functions such that $\Delta u$ is locally in $L^2(\IR^3  \setminus \Gamma)$ \cite[Chapter 4]{mclean2000strongly}.  Whenever $\gamma^+_d = \gamma^-_d$ (resp.~$\gamma^+_n = \gamma^-_n$) we denote $\gamma_d = \gamma_d^\pm$ (resp.~$\gamma_n = \gamma_n^\pm$).
%%%%%%%%%%%%%%%%%%%%%%%%%%%%%%%%%%%%%%%%%%%%%%%%%%%%%%%%%%%%
\subsection{Spherical Harmonics and Projected Basis}
\label{sec:SH}
%%%%%%%%%%%%%%%%%%%%%%%%%%%%%%%%%%%%%%%%%%%%%%%%%%%%%%%%%%%%

The standard spherical harmonics functions consist of a family of smooth functions defined in $\mathbb{S}$ given by 
\begin{align*}
Y^l_m(\theta,\varphi) := c_m\sqrt{\frac{(2l+1)(l-|m|)!}{4\pi (l+|m|)!}} \mathrm{P}^l_{|m|}(\cos\varphi)e^{im \theta}, 
\end{align*}
where $c_m = (-1)^m$ if $m<0$ or $c_m=1$ otherwise, with $(\theta,\varphi) \in [0,2\pi]\times [0,\pi]$ are the spherical coordinates on $\mathbb{S}$, $\mathrm{P}^l_m$ denotes the associated Legendre function, and the indices $l \in \IN$, $m \in \mathbb{Z} : |m|\leq l$.
We will write $Y^l_m(\vx)$  to denote the corresponding spherical harmonic evaluated at a point $\vx \in \mathbb{S}$.   
The spherical harmonics form an orthonormal basis of $L^2(\mathbb{S})$ (\emph{cf.~}\cite{macrobert1948spherical} for more details).  

Functions $Y^l_m$ can be projected onto $\mathbb{D}$, this fact being key ingredient of our approximation basis. Let us start by defining two lifting operators on $\cD(\mathbb{D})$, the first is the \emph{even lifting}: 
\begin{align*}
L_e &: \cD(\mathbb{D}) \rightarrow \cD(\mathbb{S}), \\ 
L_e(f)(\vx) &= f(x_1,x_2), \quad \vx \in \mathbb{S} \subset \IR^3,
\end{align*}
using polar-spherical coordinates we see that $L_e(g)(\theta, \varphi) = g(\sin\varphi,\theta)$, where the arguments of $g$ are the distance to the origin and the polar angle on the plane $x_3 =0$. The odd lifting has to be defined on $\cD_0(\mathbb{D})$ the space of smooth functions on $\mathbb{D}$ which vanish on the unitary circle in the plane $x_3=0$. 
\begin{align*}
L_o &: \cD_0(\mathbb{D}) \rightarrow \cD(\mathbb{S}), \\
L_o(f)(\vx) & = \begin{cases} f(x_1,x_2) \quad x_3 \geq 0\\ -f(x_1,x_2) \quad x_3 \leq 0 \end{cases} , \quad \vx \in \mathbb{S} \subset \IR^3,
\end{align*}
Notice that every function in the image of $L_e$ (resp.~$L_o$) is a even (resp.~odd) function on the $x_3$ variable. In particular, the spherical harmonic $Y^l_m$ is even (resp.~odd) when $m+l$ is even (resp.~odd). The definition of the lifting operators can be extended to distributions by duality. In parallel, one can define a projection operator as
\begin{align*}
\Pi_\mathbb{S} &: \cD(\mathbb{S}) \rightarrow \cD(\mathbb{D}), \\
\Pi_\mathbb{S}(f) (\vx) &:= f\left(\vx,  \sqrt{1-x_1^2-x_2^2}\right) \quad \vx \in \mathbb{D} \subset \IR^2,
\end{align*}
which is the inverse $L_e$ (resp.~$L_o$) when restricted to the functions with the corresponding symmetries on $\mathbb{S}$.  Finally, following \cite{wolfe1971,RamaTesis}, the projected basis are defined as 
\begin{align*}
p^l_m(\vx) := \sqrt{2}\Pi_\mathbb{S}(Y^l_m)(\vx),\quad
q^l_m(\vx) := \frac{p^l_m(\vx)}{\sqrt{1-\|\vx\|^2}},
\end{align*}
where $\vx \in \mathbb{D}$. From the orthogonality property of spherical harmonics, it holds that
\begin{align}
\label{eq:diskOrt}
\int_\mathbb{D} p^l_m(\vx) \overline{q^{l'}_{m'}(\vx)} d\vx = \delta_{m,m'}\delta_{l,l'}.
\end{align}
Also, by symmetry of spherical harmonics, one can see that if $m+l$ is even, then $p_m^l$ is a smooth function, while, if $m+l$ is odd, then $q_m^l$ is smooth.

An explicit formula for the  projected basis is given by: 
\begin{align}
\label{eq:projectedexplicit}
p^l_m(r,\theta) = C^l_{m} \mathrm{P}^l_{|m|}(\sqrt{1-r^2})e^{im\theta},
\end{align} 
where $C^l_m  = \sqrt{\frac{(2l+1)(l-|m|)!}{2\pi (l+|m|)!}}$ if $m\geq0$ and $C^l_m = (-1)^m
C^l_{-m}$ if $m < 0$.

 To simplify notations, we sporadically use only one sub-index for the projected basis. For $m+l$ even, we can reorder the basis with a one-dimensional index defined as $$I_e(l,m):=\dfrac{l(l+1)+(l+m)}{2}.$$ The even function indexed in this way will be denoted $p^e_{I_e(l,m)}$ (resp. ~$q^e_{I_e(l,m)}$) whereas for  $m+l$ odd,  we set $$I_o(l,m) = \frac{(l-1)l + m +l-1}{2},$$ with  functions denoted $p^o_{I_o(l,m)}$ (resp.~$q^o_{I_o(l,m)}$). For example, this leads to
\begin{align*}
p^e_0 &= p^0_0, &\  p^e_1 &= p^1_{-1},&\ p^e_2 &= p^1_{1},  &p^e_3 &= p^2_{-2},\\
q^o_0 &= q^1_0 &\  q^o_1 &= q^2_{-1} &\  q^o_2 &= q^2_{1}  &\ q^o_3 &= q^3_{-2}. 
\end{align*}

\begin{lemma} [Proposition 2.1.20 in \cite{SinBrillo}]
\label{lemma:densities}
The following inclusions are dense in the corresponding Sobolev spaces: 
 \begin{align*}
  \spn\{p_l^e\}_{l\in \IN} \subset {H}^{\half}(\mathbb{D}), \quad
   &  \spn\{p_l^o\}_{l\in \IN} \subset \widetilde{H}^{\half}(\mathbb{D}), \\
   \spn\{q_l^e\}_{l\in \IN} \subset \widetilde{H}^{-\half}(\mathbb{D}), \quad 
  &  \spn\{q_l^o\}_{l\in \IN} \subset {H}^{-\half}(\mathbb{D}). 
 \end{align*}
\end{lemma}

%%%%%%%%%%%%%%%%%%%%%%%%%%%%
\subsection{Auxiliary Functional Spaces}
%%%%%%%%%%%%%%%%%%%%%%%%%%%%
Classically, Sobolev spaces on $\mathbb{S}$ can be defined in terms of functions expressed as an expansion of spherical harmonics in the following way (\emph{cf.}~\cite{Pham:2011} or \cite[Chapter 7]{atkinson2001theoretical}): 
\begin{align*}
H^s(\mathbb{S}) := \left\{ u \in \mathcal{D}'(\mathbb{S}) : 
\| u\|_{H^s(\mathbb{S})}^2 := \sum_{l =0}^\infty (l+1)^{2s} \sum_{m=-l}^l |\langle u, Y^l_m\rangle |^2 < \infty \right\}, \quad s \in \IR,
\end{align*}
where $\langle u, Y^l_m\rangle$ is the extension by duality of the $L^2(\mathbb{S})$-inner product.  From this definition, given $u \in H^s(\mathbb{S})$ and $N \in \IN$,  we can define a finite-dimensional projection of $u$ as 
\begin{align*}
\Pi_N u = \sum_{l=0}^N \sum_{m=-l}^l \langle u, Y^l_m \rangle Y^l_m . 
\end{align*}
and using elementary properties we can bound the error of this projection as: 
\begin{align*}
\| u -u_N\|_{H^s(\mathbb{S})} \leq (N+1)^{s-r} \|u \|_{H^r(\mathbb{S})} ,
\end{align*}
 for any reals $s,r$ such that $s<r$. The space $H^s_e(\mathbb{S})$ (resp. $H^s_o(\mathbb{S})$) is defined as the completion of the even (resp.~odd) functions in the $x_3$ variable in $\cD(\mathbb{S})$, with the same norm. For these cases, the norms can be written as 
\begin{align*}
\|u \|_{H^s_e(\mathbb{S})}^2 = \sum_{l=0}^\infty (n+1)^{2s} \sum_{\substack{m=-l\\ m+l \text{ even}}}^l |\langle u , Y^l_m \rangle|^2, \\
\|u \|_{H^s_o(\mathbb{S})}^2 = \sum_{l=0}^\infty (n+1)^{2s} \sum_{\substack{m=-l\\ m+l \text{ odd}}}^l |\langle u , Y^l_m \rangle|^2.
\end{align*}
We will also consider the special function $w_\mathbb{S}:=\sqrt{1-x_1^2-x_2^2}$ defined in $\mathbb{S}$. In spherical coordinates can be written as $w_\mathbb{S} = |\cos\varphi|$.  
\begin{remark}
As the function $w_\mathbb{S}$ do not depend on the $x_3$ coordinate, it holds that
\begin{align*}
w_\mathbb{S} L_e(u) = L_e(w_\mathbb{D} u), \quad 
w_\mathbb{S} L_o(u) = L_o(w_\mathbb{D} u),
\end{align*}
where $w_\mathbb{D}(\vx)= \sqrt{1-\|\vx\|^2}$, for $\vx \in \mathbb{D}$.  
\end{remark} 

Let us introduce two families of auxiliary spaces on $\mathbb{D}$ associated with even functions: 
\begin{align*}
P^s_e(\mathbb{D}) &:=\{ u \in \cD'(\mathbb{D}) : L_e(u) \in H^s_e(\mathbb{S}) \} \\
Q^s_e(\mathbb{D}) &:= \{u \in \cD'(\mathbb{D}) : w_\mathbb{S} L_e(u) \in H^s_e(\mathbb{S}) \},
\end{align*}
with norms given by: 
\begin{align*}
\|u\|_{P^s_e(\mathbb{D})} &:= \| L_e(u) \|_{H^s_e(\mathbb{S})} = \left( \sum_{l=0}^\infty (n+1)^{2s} \sum_{\substack{m=-l\\ m+l \text{ even}}}^l |\langle u , q^l_m \rangle_{\mathbb{D}}|^2
\right) ^{\half},\\
\|u\|_{Q^s_e(\mathbb{D})} &:= \| w_\mathbb{S}L_e(u) \|_{H^s_e(\mathbb{S})} = \left( \sum_{l=0}^\infty (n+1)^{2s} \sum_{\substack{m=-l\\ m+l \text{ even}}}^l |\langle u , p^l_m \rangle_{\mathbb{D}}|^2
\right) ^{\half}.
\end{align*}
Odd function spaces $P^s_o(\mathbb{D})$, $Q^s_o(\mathbb{D})$ are defined in a similar fashion. While the connection with standard spaces is not as direct as the definition suggests, the next Lemma will be useful.\footnote{The proof is given in Appendix \ref{proof:lemma:sobRev}.}
\begin{lemma} 
\label{lemma:sobRev}
The following relations between auxiliary and classical spaces on $\mathbb{D}$ holds: 
\begin{align}
\label{lemsobRev1}P^0_e(\mathbb{D}) & = L^2_{1/w_\mathbb{D}}(\mathbb{D}):= \left\{  u : \int_\mathbb{D}  \frac{u\overline{u}}{w_\mathbb{D}} d\vx < \infty \right\}\\
\label{lemsobRev2}Q^{-\fourth}_e(\mathbb{D}) &\subset \widetilde{H}^{-\half}(\mathbb{D}) \subset Q^{-\half}_e(\mathbb{D})\\
\label{lemsobRev3}P^{\half}_o(\mathbb{D}) &\subset \widetilde{H}^{\half}(\mathbb{D}) \subset P^{\fourth}_o(\mathbb{D}). 
\end{align}
with continuous inclusions. 
\end{lemma} 

\begin{corollary}
The following inclusions are continuous: 
\begin{align}
Q_o^{-\fourth}(\mathbb{D}) &\subset H^{-\half}(\mathbb{D}) \subset Q_o^{-\half}(\mathbb{D})\\
P_e^{\half}(\mathbb{D}) &\subset H^{\half}(\mathbb{D}) \subset P_e^{\fourth}(\mathbb{D}).
\end{align}
\end{corollary}
\begin{proof}
Both results are immediate consequences of the duality relation between classical Sobolev spaces and Lemma \ref{lemma:sobRev}. 
\end{proof}

By definition of the auxiliary spaces as well the projector $\Pi_N$, we can introduce four projectors onto the disk. Namely, 
%\begin{align*} 
%\Pi^{e,p}_N : \sum_{l =0}^\infty  \sum_{\substack{m=-l\\ m+l \text{ even}}}^l  u^l_m p^l_m \mapsto 
%\sum_{l =0}^N  \sum_{\substack{m=-l\\ m+l \text{ even}}}^l  u^l_m p^l_m 
%\end{align*}
\begin{align*} 
\Pi^{e,p}_N : \sum_{l =0}^\infty  u_l p^e_l \mapsto 
\sum_{l =0}^{I(N)}    u_l p^e_l ,\quad &
\Pi^{o,p}_N : \sum_{l =0}^\infty  u_l p^o_l \mapsto 
\sum_{l =0}^{I(N)}    u_l p^o_l ,\\
\Pi^{e,q}_N : \sum_{l =0}^\infty  u_l q^e_l \mapsto 
\sum_{l =0}^{I(N)}    u_l q^e_l ,\quad &
\Pi^{o,q}_N : \sum_{l =0}^\infty  u_l q^o_l \mapsto 
\sum_{l =0}^{I(N)}    u_l q^o_l,
\end{align*}
where $I(N)= I_e(N,N) = \frac{N(N+1)}{2} +N$. Using the norm definitions on the corresponding spaces, it holds that
\begin{align} 
\label{eq:ProjEstimations}
\| \Pi^{e,p}_N u - u \|_{P^s_e(\mathbb{D})} \leq (N+1)^{s-r} \| u \|_{P^r_e(\mathbb{D})}, \\
\| \Pi^{o,p}_N u - u \|_{P^s_o(\mathbb{D})} \leq (N+1)^{s-r} \| u \|_{P^r_o(\mathbb{D})}, \\
\| \Pi^{e,q}_N u - u \|_{Q^s_e(\mathbb{D})} \leq (N+1)^{s-r} \| u \|_{Q^r_e(\mathbb{D})}, \\
\| \Pi^{o,q}_N u - u \|_{Q^s_e(\mathbb{D})} \leq (N+1)^{s-r} \| u \|_{Q^r_o(\mathbb{D})},
\end{align}
for any reals $s,r$ such that $s< r$.

%%%%%%%%%%%%%%%%%%%%%%%%%%%%%%%%%%%%
\section{Boundary Integral Formulation}
%%%%%%%%%%%%%%%%%%%%%%%%%%%%%%%%%%%%
We now reduce the original problem to the screen.

\subsection{Boundary Integral Operators}
\label{sec:BIOS}
Let us recall the definitions of the single and double layer potentials: 
\begin{align*}
(\mathcal{S}_\Gamma[k]\lambda)(\vx) := \int_\Gamma G_k(\vx,\vx')\lambda(\vx')d\vx', \quad \vx\in \IR^3\setminus \overline{\Gamma}, \\
(\mathcal{D}_\Gamma[k]\nu)(\vx) := \int_\Gamma \gamma_{n,\vx'}G_k(\vx,\vx')\nu(\vx')d\vx', \quad \vx\in \IR^3\setminus \overline{\Gamma}, 
\end{align*}
respectively, where $$G_k(\vx,\vx') = \frac{e^{ik\|\vx-\vx'\|}}{4 \pi \| \vx-\vx'\|}$$ denotes the Helmholtz fundamental solution and $\gamma_{n,\vx'}$ corresponds to the Neumann trace in the specified variable. 

Following the clasical formulation of Dirichlet (resp.~Neumann) problems, we can seek for solutions of the form $u(\vx) = (\mathcal{S}_\Gamma[k]\lambda)(\vx)$ (resp. $u(\vx) = (\mathcal{D}_\Gamma[k]\nu)(\vx)$), where $\lambda$ (resp.~$\nu$) is an unknown density defined on $\Gamma$. By taking traces, one naturally defines the weakly- and hyper-singular boundary integral operators (BIOs),
\begin{align*}
(V_\Gamma [k] \lambda)(\vx)   &:= \gamma_d (\mathcal{S}_\Gamma[k]\lambda)(\vx) = \int_\Gamma G_k(\vx,\vx')\lambda(\vx')d\vx', \quad x\in \Gamma, \\
(W_\Gamma [k] \nu)(\vx)   &:= -\gamma_n (\mathcal{D}_\Gamma[k]\nu)(\vx) = - \gamma_{n}\int_\Gamma  \gamma_{n,\vx'}G_k(\vx,\vx')\nu(\vx')d\vx', \quad x\in \Gamma, 
\end{align*}
where first BIO is defined as a Lebesgue integral, while the second one is understood as a principal value \cite[Chapter 5]{mclean2000strongly}. 
%%%%%%%%%%%%%%%%%%%%%%%%%%%%%%%%%%%
\subsection{Boundary Integral Equations}
%%%%%%%%%%%%%%%%%%%%%%%%%%%%%%%%%%%
With the above definitions, Problem \ref{prob:edp} can be reduced to 

\begin{problem}[BIEs]\label{prob:BIEs}
For $k\geq0$, find $\lambda, \nu$ $\in \widetilde{H}^{-1/2}(\Gamma)\times \widetilde{H}^{1/2}(\Gamma)$ such that 
\begin{align}
\label{eq:bies}
V_\Gamma[k] \lambda  = g_d, \quad \text{(Dirichlet BIE)},\\
W_\Gamma[k] \nu  = g_n, \quad \text{(Neumann BIE)}.
\end{align}
\end{problem}

The equivalence between these BIEs and their corresponding original problems is established in \cite{Stephan86}. 

\begin{theorem}[Theorem 2.7, Lemma 2.8, in \cite{STE86}, Theorem 2.1.60  in  \cite{Sauter:2011}]
For any $k \geq 0$, $g_d \in H^{\half}(\Gamma)$ (resp.~$g_n \in H^{-\half}(\Gamma)$), there exists one solution 
$\lambda \in \widetilde{H}^{-\half}(\Gamma)$ (resp.~$\nu \in \widetilde{H}^{\half}(\Gamma)$) for Problem \ref{prob:BIEs}. Moreover, the solution operators are bounded:
\begin{align*}
\| \lambda \|_{\widetilde{H}^{-\half}(\Gamma)} \lesssim \| g_d\|_{H^{\half}(\Gamma)}, \quad
\| \nu \|_{\widetilde{H}^{\half}(\Gamma)} \lesssim \| g_n\|_{H^{-\half}(\Gamma)},  
\end{align*}
with unspecified constants depending on $\Gamma$ and $k$.  
\end{theorem}

\begin{theorem}[Corollary A.4 in \cite{costabel2003}]
\label{thrm:solreg}
Given $g_d$, $g_n$ smooth functions on $\Gamma$, the pulled-back solutions multiplied by the corresponding weight functions, $(\lambda  \circ \vr) w_\mathbb{D}$ and $(\nu \circ \vr)(w_\mathbb{D})^{-1}$, are as smooth as $g_d \circ \vr$, and $g_n \circ \vr$, respectively. 
\end{theorem}

\begin{remark}
The last theorem is far from trivial. In the two-dimensional case --open arcs-- solutions are also singular at the arc endpoints, which is proven straightforwardly for any $k$ and arc as a perturbation of the simpler case $k= 0$, $\Gamma= (-1,1)\times \{0\}$. On the other hand, for screens in three-dimensional space a careful analysis of the BIO symbols is needed (\emph{cf.}~\cite{costabel2003} and references therein). 
\end{remark} 

%%%%%%%%%%%%%%%%%%%%%%%%%%%%%%%%
\section{Spectral Discretizations}
%%%%%%%%%%%%%%%%%%%%%%%%%%%%%%%%
From the definition of auxiliary spaces Lemma \ref{lemma:densities}, it is natural to consider a collection of finite-dimensional spaces spanned by elements $q_l^e$ (resp.~$p_l^o$) for the discretization of the Dirichlet (resp.~Neumann) BIEs.  

For $N \in \IN$, we set the following spaces defined over the disk:
\begin{align*}
\mathbb{Q}^N_e(\mathbb{D}) &:= \spn\{q^e_l\}_{l=0}^{I(N)} = \left\lbrace 
u  : \ u = 
\sum_{l=0}^N \sum_{\substack{m=-l\\ m+l \text{ even}}}^l u^l_m q^l_m, \ u^l_m \in \mathbb{C}
\right\rbrace  \subset \widetilde{H}^{-\half}(\mathbb{D}), \\
\mathbb{P}^N_o(\mathbb{D}) &:= \spn\{p^o_l\}_{l=0}^{I(N)} = \left\lbrace 
u : \ u = 
\sum_{l=0}^{N+
1} \sum_{\substack{m=-l\\ m+l \text{ odd}}}^l u^l_m p^l_m, \ u^l_m \in \mathbb{C}
\right\rbrace  \subset \widetilde{H}^{\half}(\mathbb{D}), 
\end{align*}
where again $I(N) = \frac{N(N+1)}{2}+N$. For an arbitrary smooth screen, we define corresponding spaces through pullbacks as follows
\begin{align*}
\mathbb{Q}^N_e(\Gamma) &:= \left\lbrace u : \  = \frac{(v \circ \vr^{-1}) \| \vr^{-1}\|}{J_\vr \circ \vr^{-1}}, \quad v \in \mathbb{Q}^N_e(\mathbb{D}) \right\rbrace \subset \widetilde{H}^{-\half}(\Gamma), \\
\mathbb{P}^N_o(\Gamma) &:= \left\lbrace u : \  = v \circ \vr^{-1}, \quad v \in \mathbb{P}^N_o(\mathbb{D}) \right\rbrace \subset \widetilde{H}^{\half}(\Gamma).
\end{align*}
where the inclusions can be easily shown. These spaces are spanned by the basis functions:
\begin{align*}
q^{e,r}_n &:= \frac{(q^e_n \circ \vr^{-1}) \| \vr^{-1}\|}{J_\vr \circ \vr^{-1}}, \quad
p^{o,r}_n := p^o_n \circ \vr^{-1}, 
\end{align*}
respectively, and where $n = 0,\ldots,I(N)$.

%%%%%%%%%%%%%%%%%%%%%%%%%%%%%%%%
\subsection{Discrete Problem}
\label{sec:DiscreteProblems}
%%%%%%%%%%%%%%%%%%%%%%%%%%%%%%%%
Before we introduce discrete versions of Problem \ref{prob:BIEs}, we set forth the following notation. Given $N \in \mathbb{N}$, vectors on  $\mathbb{C}^{I(N)+1}$ are written in bold symbols and superindex $N$. Associated to every vector, there are two functions which are denoted with the same symbol (not in bold) and superindex $N$, and an extra the sub-index $e$ if the function is to be understood in $\mathbb{Q}^N_e(\Gamma)$, and $o$ if in $\mathbb{P}_o^N(\Gamma)$. For example, given $\vlambda^N \in \mathbb{C}^{I(N)+1}$, one can write
\begin{align*}
\lambda^N_e = \sum_{m=0}^{I(N)} \lambda^N_m q_m^{e,r},\quad 
\lambda^N_o = \sum_{m=0}^{I(N)} \lambda^N_m p_m^{o,r}.
\end{align*}
Given $N \in \IN$, the Galerkin discretization of the BIE formulation \eqref{eq:bies} reads as 
\begin{problem}
\label{prob:discreteBIEs}
Seek  $\mathbf{\vlambda}^N \in \mathbb{C}^{I(N)+1}$ (resp.~$\vnu^N \in \mathbb{C}^{I(N)+1}$) such that 
\begin{align*}
V^N_\Gamma[k] \vlambda^N = \vg_d^N, \quad \text{(discrete Dirichlet BIE)}.\\
W^N_\Gamma[k] \vnu^N = \vg_n^N,\quad \text{(discrete Neumann BIE)},
\end{align*}
where the respective discretization matrices elements are defined as 
\begin{align*}
(V^N_\Gamma[k])_{l,m} = \langle V_\Gamma[k] q^{e,r}_m, q_l^{e,r}\rangle_\Gamma, \quad 
(W^N_\Gamma[k])_{l,n} = \langle W_\Gamma[k] p^{o,r}_m, p_l^{o,r}\rangle_\Gamma,
\end{align*}
and the corresponding discrete right-hand sides are, for $l,m =0,\hdots,I(N)$,
\begin{align*}
(g_d^N)_l = \langle g_d, q^{e,r}_l \rangle_\Gamma, \quad 
(g_n^N)_l=\langle g_n, p^{o,r}_l \rangle_\Gamma.
\end{align*}
\end{problem}
The discrete approximations for $\lambda, \nu$ (solutions of Problem \ref{prob:BIEs}) are $\lambda^N_e, \nu^N_o$ respectively. 
%\cj{Discrete} solutions are constructed as 
%\begin{align*}
%\lambda_N = \sum_{n=0}^{I(N)} \lambda^N_\jp{m} q^{e,r}_\jp{m}, \quad
%\nu_N = \sum_{n=0}^{I(N)} \nu^N_\jp{m} p^{o,r}_\jp{m}.
%\end{align*}
From standard Galerkin properties we have the following result:
\begin{lemma}[Theorem 4.29 in \cite{Sauter:2011}]
\label{lemma:qopti}
There exists $N_0 \in \IN_0$ --potentially different for the weakly- and hyper-singular BIEs-- such that, for any $N \in \IN$ with $N> N_0$, the solutions $\vlambda^N$, and $\vnu^N$ of Problem \ref{prob:discreteBIEs} exist, are unique and also we obtain quasi-optimality: 
\begin{align*}
\| \lambda - \lambda^N_e \|_{\widetilde{H}^{-\half}(\Gamma)} &\lesssim \inf_{v \in \mathbb{Q}^e_N(\Gamma)}
\| \lambda - v \|_{\widetilde{H}^{-\half}(\Gamma)}, \quad
\| \nu - \nu^N_o \|_{\widetilde{H}^{\half}(\Gamma)} \lesssim \inf_{v \in \mathbb{P}^o_N(\Gamma)}
\| \nu - v \|_{\widetilde{H}^{\half}(\Gamma)}.
\end{align*}
\end{lemma} 
From these last estimates we obtain the rate of error convergence. 

\begin{theorem}
\label{thrm:rateConvergence}
Given $g_d, g_n \in \cD(\Gamma)$, then for any regularity indices $s_d > -\frac{1}{4}$, $s_n > \frac{1}{2}$, it holds that
\begin{align}
\label{eq:rateConvergence}
 \| \lambda - \lambda^N_e \| _{\widetilde{H}^{-\half}(\Gamma)} &\lesssim N^{-\fourth-s_d} \| \lambda \circ \vr \|_{Q^{s_d}_e(\mathbb{D})}, \\ 
   \| \nu - \nu^N_o \| _{\widetilde{H}^{\half}(\Gamma)} &\lesssim N^{\half-s_n} \| \nu \circ \vr \|_{P^{s_n}_o(\mathbb{D})}.
\end{align} 
\end{theorem}

\begin{proof} 
We prove only the Dirichlet case as the Neumann follows similar arguments. 
Denote by $R$ a smooth non-zero function $R(\vx) := \dfrac{\| \vx\|}{J_\vr(\vx)}$. By using Lemmas \ref{lemma:SobolevLocalization} and \ref{lemma:qopti} one finds that
\begin{align*}
 \| \lambda - \lambda^N_e \|_{\widetilde{H}^{-\half}(\Gamma)} &\lesssim 
  \inf_{v \in \mathbb{Q}^e_N(\Gamma)} \| \lambda - v \|_{\widetilde{H}^{-\half}(\Gamma)}
  =
  \inf_{v \in \mathbb{Q}^e_N(\mathbb{D})} \| \lambda \circ \vr - Rv \|_{\widetilde{H}^{-\half}(\mathbb{D})}.
\end{align*}  
We can use the relation between Sobolev and auxiliary spaces (Lemma \ref{lemma:sobRev}) to obtain 
\begin{align*}
 \| \lambda - \lambda^N_e \|_{\widetilde{H}^{-\half}(\Gamma)} &\lesssim \inf_{v \in \mathbb{Q}^e_N(\mathbb{D})} \left\lVert \lambda \circ \vr - Rv \right\rVert_{Q^{-\fourth}_e(\mathbb{D})}  \lesssim 
 \inf_{v \in \mathbb{Q}^e_N(\mathbb{D})} \left\lVert \frac{\lambda }{R} - v \right\rVert_{Q^{-\fourth}_e(\mathbb{D})}.
\end{align*}
Since $R^{-1}$ is smooth, the results follow by selecting $v = \Pi^{e,q}_N \frac{\lambda}{R}$ and estimate \eqref{eq:ProjEstimations} with $s = -1/4, r = s_d$.
\end{proof} 

\begin{remark}
Notice that the right-hand side in \eqref{eq:rateConvergence} is indeed finite for any $s_d>-\fourth$, $s_n>\half$. This follows from the norm definitions as
$$\| \lambda \circ \vr \|_{Q^{s_d}_e(\mathbb{D})} = 
\| L_e(w_\mathbb{D} \lambda \circ \vr) \|_{H^{s_d}(\mathbb{S})},$$
and by Theorem \ref{thrm:solreg}, which ensures that $w_\mathbb{D} \lambda \circ \vr$ is smooth. Thus, implying a smooth lifting to the sphere. Using definitions of the Sobolev spaces \cite[Chapter 2]{mclean2000strongly}, any Sobolev norm is finite. The same is true for the Neumann case. 
\end{remark}

One concludes from Theorem \ref{thrm:rateConvergence} that the spectral method converges \emph{super algebraically}, i.e.~faster than any fixed negative power of $N$. 
%%%%%%%%%%%%%%%%%%%%%%%%%%%%%%%%%%%%
\subsection{Matrix Computation}
\label{sec:MatrixComputations}
%%%%%%%%%%%%%%%%%%%%%%%%%%%%%%%%%%%%
We now describe how matrix entries are computed. We start by detailing the approximation of weakly-singular integrals that appear in the corresponding matrix, namely
\begin{align*}
V_{l,m}[k] = \langle V_\Gamma[k] q^{e,r}_m, q_l^{e,r}\rangle_{\Gamma}.
\end{align*} 
Then, we briefly discuss how integrals for the hyper-singular BIO are obtained using the same techniques, and also how regular entries are computed. Before we proceed, we recall some notions of numerical quadrature and convergence.
\subsubsection{Numerical Quadrature}

Let $a<b$ two real numbers, and $f: (a,b) \rightarrow \mathbb{C}$. The Gauss-Legendre quadrature rule approximate the integral of $f$ as a weighted sum of point evaluations of $f$, the approximation is constructed as
\begin{align*} 
\int_{a}^b f(x) dx \approx \sum_{i=1}^{N_q} w^L_i f(x^L_i),
\end{align*} 
where $N_q$ is the order of the quadrature, and $(\{w_i^L\}_{i=1}^{N_q}$ ,$\{x_i^L\}_{i=1}^{N_q})$ are the weights and points \footnote{Notice that the points and weights depend on $a,b$, but once given for a fixed interval they can be translated using a linear change of variable. Consequently, we omit this dependence in notation. } of the quadrature respectively. 
When $f$ is smooth, the quadrature converges with a rate bounded as a function of $N_q$. In particular, by using the fact that the quadrature rule is constructed to be exact for every polynomial up to some degree and classical bounds for polynomial interpolation \cite[Chap.~7 and 8]{trefethen2013approximation}, one can establish that for $f$ with $(m+1)$ continuous derivatives\footnote{Better convergence rates can be achieved for Gaussian quadrature rules \cite[Theorem 3.6.24]{stoer1980introduction}.}:
\begin{align*}
\left| \int_{a}^b f(x) dx -\sum_{i=1}^{N_q} w^L_i f(x^L_i) \right| \lesssim N_q^{-m}.
\end{align*}
Moreover, if $f$ admits and analytic extension to a Bernstein ellipse of parameter $\rho$ in the complex plane, one can show that
\begin{align*}
\left| \int_{a}^b f(x) dx -\sum_{i=1}^{N_q} w^L_i f(x^L_i) \right| \lesssim \rho^{-N_q}.
\end{align*}
Whenever an integral of a function can be approximated with the same rates as the last two, we say that the approximation is optimal. In particular, Gauss-Legendre quadrature rules are
optimal for any smooth function integral. Jacobi quadrature rules are built as an approximation of the following family of integrals: 
\begin{align*}
\int_{a}^b f(x) (x-a)^\alpha (b-x)^\beta dx \approx \sum_{i=1}^{N_q} w^{\alpha,\beta}_i f(x^{\alpha,\beta}_i),
\end{align*} 
where, again, $N_q$ is the quadrature order,  $(\{w_i^{\alpha, \beta}\}_{i=1}^{N_q}$ ,$\{x_i^{\alpha, \beta}\}_{i=1}^{N_q})$ are the pair of weights and points, and $\alpha ,\beta > -1$.  This rule is also optimal, i.e.~that the rate of convergence is again  $N_q^{-m}$ when $f$ has $(m+1)$ continuous derivative and $\rho^{-N_q}$ when $f$ is $\rho-$analytic.
%%%%%%%%%%%%%%%%%%%%%%%%%%%%%%%%%%%%%%%%%%%%%%%%
\subsubsection{Approximation of weakly-singular integrals}
\label{sec:AproxSingInt}
%%%%%%%%%%%%%%%%%%%%%%%%%%%%%%%%%%%%%%%%%%%%%%%%
Given a screen $\Gamma$ para-metrized by a function 
$ \vr $, we consider the computation of integrals of the form:
$$ 
I_{m,l}[k] = \int_\Gamma \int_\Gamma \frac{\cos(k \| \vx- \vx'\|)}{4 \pi \| \vx - \vx' \|} 
q_m^{e,r}(\vx') \overline{q_l^{e,r}(\vx)} d\vx d\vx'.
$$
These integrals are associated with the real part of the weakly-singular BIO. Its imaginary part is regular since the function $\sin(k \| \vx-\vx'\|)/\| \vx -\vx' \|$ is smooth. Moreover, the cosine factor is smooth so from here onwards, we assume that $k=0$ and denote $I_{m,l}[0]$, as $I_{m,l}$. By performing a change of variable, this integral becomes
$$ 
I_{m,l} = \int_\mathbb{D} \int_\mathbb{D} \frac{1}{4 \pi \| \vr(\vx) - \vr(\vx') \|} 
q^e_m(\vx') \overline{q^e_l(\vx)} d\vx d\vx'.
$$
The first step is to take care of the kernel singularity, i.e.~when $\vx = \vx'$. To this end, we make the following parametrization: 
\begin{align}
\label{eq:polars}
\vx = r\ve_\theta, \quad 
\vx' = \vx + \lambda A(r,\beta) \ve_{\theta+\beta}, 
\end{align} 
where
 \begin{align}
\label{eq:Afun}
A(r,\beta) := \sqrt{1-r^2\sin^2\beta}- r\cos\beta,
\end{align}
represents the length of the segment whose direction is $\theta +\beta$, and goes from point $\vx$ to the boundary of the disk (see Figure \ref{fig:polar}).
\begin{figure}
\centering
\includegraphics[scale=0.1]{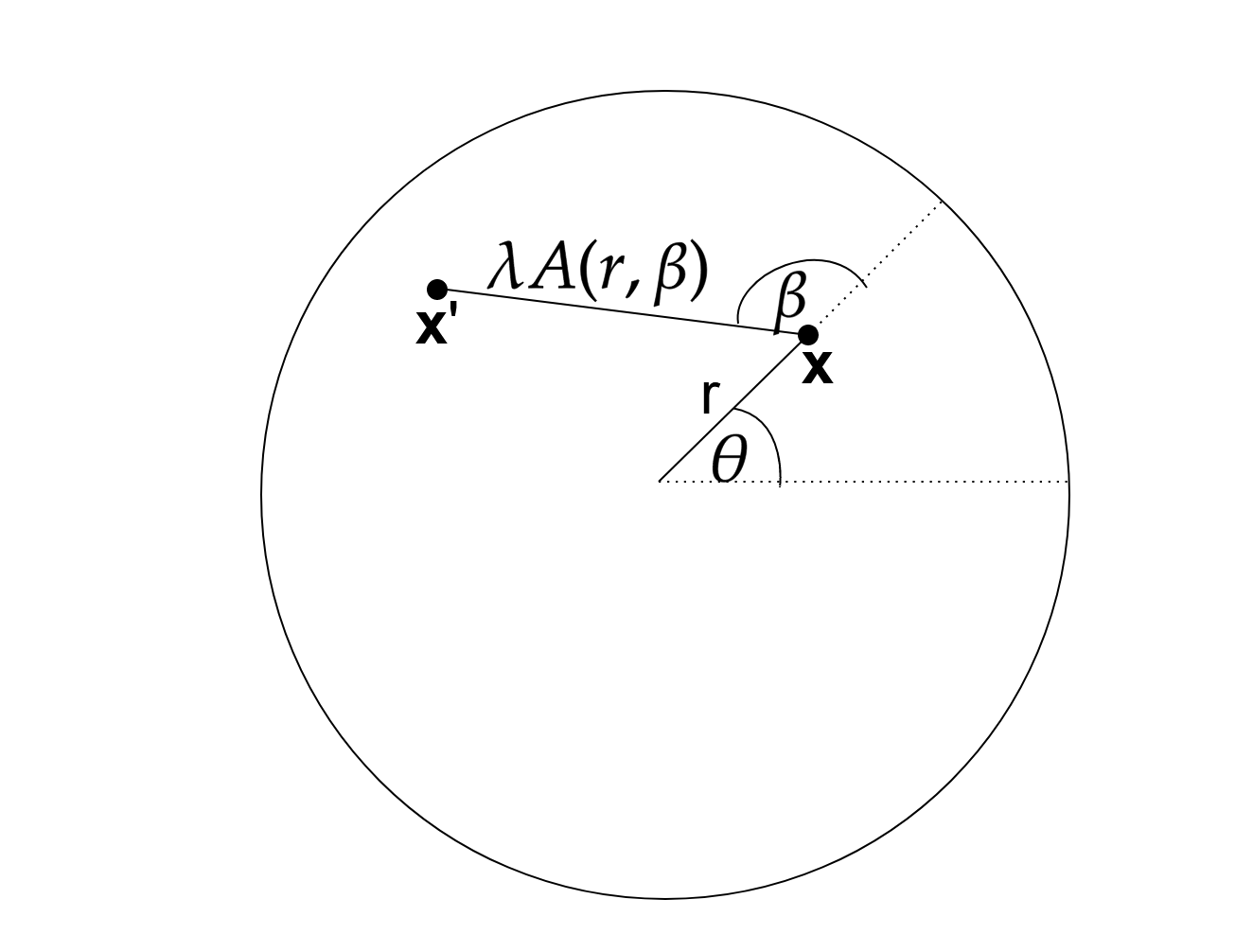}
\label{fig:polar}
\caption{Polar change of variables performed in \eqref{eq:polars}.} 
\end{figure}
The integral can be expressed as 
\begin{align*}
I_{m,l} &= \int_{-\frac{\pi}{2}}^{3\frac{\pi}{2}} \int_0^1 \int_{-\frac{\pi}{2}}^{3\frac{\pi}{2}} \int_{0}^{1} 
G_{m,l}(\theta,r,\lambda,\beta)d\lambda d\beta dr d\theta ,\\
G_{m,l}(\theta,r,\lambda,\beta)&:= 
\frac{ \lambda A^2(r,\beta) }
{ 4 \pi \|\vr(\vx) - \vr(\vx + \lambda A(r,\beta) \ve_{\theta+\beta}) \|}
\frac{r P^e_m(\vx) \overline{P^e_l(\vx')}}{\sqrt{1-r^2}\sqrt{1- \| \vx +\lambda A(r,\beta)\ve_{\theta+\beta} \|^2}
}
.
\end{align*}
Since $\vr$ is smooth and its gradient has full rank, the factor
$$
\frac{ \lambda A^2(r,\beta) }
{ 4 \pi \|\vr(\vx) - \vr(\vx + \lambda A(r,\beta) \ve_{\theta+\beta}) \|},
$$
is at least bounded. The term $(1-r^2)^{-\half}$ while singular can be tackled with the Jacobi rule. However, this is not enough since $A(r,\beta)$ and$\sqrt{1- \| \vx +\lambda A(r,\beta)\ve_{\theta+\beta} \|^2}$ also have singularities that prevent an optimal rate of convergence.  The following results characterize the behavior of these functions. 
\begin{lemma}
\label{lemma:Asings}
The function $A : [0,1] \times  [\frac{-\pi}{2},\frac{3\pi}{2}] \rightarrow [0,2]	$ has the following properties: 
\begin{enumerate}
\item[(i)]  $A(r,\beta) \geq 0$.
\item[(ii)] Partial derivatives discontinuities of $A(r,\beta)$ are located at $\{r = 1 , \beta = \frac{-\pi}{2}, \frac{\pi}{2}, \frac{3\pi}{2} \}$.
\item[(iii)] For $|\beta| \leq \frac{\pi}{2}$, $A(r,\beta) \cos \beta \leq 1-r$.  
\end{enumerate}
\end{lemma}

\begin{proof}
The first item is immediate by definition. For the second one, notice that the discontinuities occur when the square-root parts vanish, i.e.
$$ (1-r^2) +r^2\cos^2\beta = 0.$$
As a sum of two non negative terms, the singularities occurs when both terms vanish, thus leading to our result directly. 
The last item is obtained by direct evaluation at extreme points $\beta = \pm \frac{\pi}{2}$ and critical values of $\beta$ such that $\dfrac{\partial A(r,\beta)\cos\beta}{\partial \beta} =0$, or by directly using elementary geometrical proprieties. 
\end{proof} 

\begin{lemma}
\label{lemma:pesosing}
Discontinuities of any partial derivative of the term
$$
\sqrt{1- \| \vx +\lambda A(r,\beta)\ve_{\theta+\beta} \|^2},
$$
occur at $\{\lambda =1\} \cup \{ r= 1, \lambda =0\} \cup 
\{ r=1, \beta = \pm \frac{\pi}{2},\frac{3\pi}{2} \}$.
\end{lemma}
\begin{proof}
Again, critical points occur only when the term inside the square-root vanishes. This term can be expressed as 
\begin{align}
\label{eq:pesoexp}
1-\| \vx +\lambda A(r,\beta)\ve_{\theta+\beta} \|^2 = 
(1-\lambda)\left[ (1+\lambda)(1-r^2) -2\lambda r \cos(\beta) A(r,\beta) \right],
\end{align} 
for which all previously listed points are critical. Hence, we are left to check that no other critical points exist. 

First, let us consider the case $ \frac{\pi}{2} \leq \beta \leq \frac{3 \pi}{2}$ so that  $-2\lambda r \cos\beta A(r,\beta) \geq 0$. Thus, the singularities are characterized by $\lambda =1$ --because of the first factor--, and also the points where 
\begin{align*}
(1+\lambda)(1-r^2) = 0, \\
2\lambda r \cos\beta A(r,\beta) =0.
\end{align*}
From this condition and Lemma \ref{lemma:Asings} it is easy to see that no further critical points take place. Now, if $|\beta| \leq \frac{\pi}{2}$, by the third item of Lemma \ref{lemma:Asings} we have that 
\begin{align*}
( (1+\lambda)(1-r^2) -2\lambda r \cos\beta A(r,\beta) ) \geq 
(1+\lambda) (1-r^2) -2\lambda r (1-r)), 
\end{align*} 
which implies that the singularities can only occur if $\lambda = 1$ --because of the first term--, or $r=1$, and the result follows. 
\end{proof}

Based on the above considerations, we split the integrals $I_{m,l}$ according to the singularities and critical points as follows. The detailed analysis for the last three types is given in the Appendix \ref{Appendix:quads}.
\begin{itemize}
\item[(a)] The first type is
\begin{align*}
I^{a}_{m,l} := \int_{-\frac{\pi}{2}}^{\frac{3\pi}{2}} \int_{0}^{\frac{\sqrt{3}}{2}} \int_{-\frac{\pi}{2}}^{\frac{3\pi}{2}} \int_0^1  G_{m,l}(\theta,r,\lambda,\beta) d\lambda d\beta dr d \theta,
\end{align*}
for which only a singularity of the form $(1-\lambda)^{-\half}$ occurs. This can numerically be treated using the corresponding Jacobi rule when integrating in $\lambda$ and Gauss-Legendre for the integrals in $\theta,r,\beta$, resulting in an optimal approximation.
\item[(b)] The second type is
\begin{align*}
I^{b}_{m,l} := \int_{-\frac{\pi}{2}}^{\frac{3\pi}{2}} \int_{\frac{\sqrt{3}}{2}}^1 \int_{(\frac{-\pi}{3},\frac{\pi}{3}) \cup 
(\frac{2\pi}{3},\frac{4\pi}{3}) } \int_{\half}^1  G_{m,l}(\theta,r,\lambda,\beta) d\lambda d\beta dr d \theta,
\end{align*}
which has singularities of the form $(1-r)^{-\half}$, and $(1-\lambda)^{-\half}$, so we use a Jacobi rule in these two variables and Gauss-Legendre for $\theta$ and $\beta$. 
\item[(c)] The third one has critical points occurring as a combination of variables:
\begin{align}
\label{eq:I1integral}
I^{c}_{m,l}:= \int_{-\frac{\pi}{2}}^{\frac{3 \pi}{2}} \int_{\frac{\sqrt{3}}{2}}^1 \int_{(\frac{-\pi}{2},\frac{-\pi}{3}) \cup 
(\frac{\pi}{3},\frac{\pi}{2})\cup (\frac{\pi}{2},\frac{2\pi}{3})
\cup (\frac{4\pi}{3},\frac{3\pi}{2})   } \int_{0}^{\half}  G_{m,l}(\theta,r,\lambda,\beta) d\lambda d\beta dr d \theta,
\end{align}
Here, critical points lie in $(\lambda,r,\beta) = (0,1,\beta)$, and $(\lambda,r,\beta) = (\lambda,1,\beta_0)$, $\beta_0 \in \{\pm \frac{\pi}{2}, \frac{3\pi}{2} \}$. To tackle this, we will use two polar change of variables (see \ref{sec:Ic}).

\item[(d)] The fourth one is of the form 
\begin{align}
\label{eq:I2integral}
I^d_{m,l} := \int_{-\frac{\pi}{2}}^{\frac{3\pi}{2}} \int_{\frac{\sqrt{3}}{2}}^1 \int_{(\frac{-\pi}{3},\frac{\pi}{3}) \cup 
(\frac{2\pi}{3},\frac{4\pi}{3})} \int_{0}^{\half}  G_{m,l}(\theta,r,\lambda,\beta) d\lambda d\beta dr d \theta,
\end{align} 
being sightly simpler that the previous case and only requiring one polar change of variable in $\lambda$ and $r$ variables. 
\item[(e)] Finally, the last integral type is given by
\begin{align}
\label{eq:I3integral}
I^e_{m,l} := \int_{-\frac{\pi}{2}}^{\frac{3\pi}{2}} \int_{\frac{\sqrt{3}}{2}}^1 \int_{(\frac{-\pi}{2},\frac{-\pi}{3}) \cup 
(\frac{\pi}{3},\frac{\pi}{2})\cup (\frac{\pi}{2},\frac{2\pi}{3})
\cup (\frac{4\pi}{3},\frac{3\pi}{2})   } \int_{\half}^1  G_{m,l}(\theta,r,\lambda,\beta) d\lambda d\beta dr d \theta,
\end{align}
which needs one polar change of variable in $(r,\cos\beta)$ and application of a Jacobi rule for the integral in the $\lambda$ variable.
\end{itemize}

%%%%%%%%%%%%%%%%%%%%%%%%%%%%%%%%%%%%%%%%%%%%
\subsubsection{Approximation of smooth integrals}
%%%%%%%%%%%%%%%%%%%%%%%%%%%%%%%%%%%%%%%%%%%%
Smooth integrals can be of two forms: 
\begin{align*}
I^f_{m,l}&:= \int_\Gamma \int_\Gamma {G_{reg}}(\vx,\vx') q^{e,r}_m(\vx') \overline{q_l^{e,r}(\vx)} d\vx d\vx',\quad
I^g_{m,l}:= \int_\Gamma g(\vx) \overline{q^{e,r}_l(\vx)} d\vx,
\end{align*}
where ${G_{reg}}$, and $g$ are smooth functions. The former comes from the imaginary part of the fundamental solution, whereas the latter from testing the right-hand side. We focus only in the second case, as the first one is just a tensorisation. 
By using the screen parametrization, we obtain
\begin{align*}
I^g_{m,l} =  \int_{-\frac{\pi}{2}}^{\frac{3\pi}{2}}\int_0^1 g(\vr(\vx))  \frac{\overline{p^e_l(\vx)}}{\sqrt{1-r^2}} r dr d\theta,
\end{align*}
where $\vx = r\ve_\theta$. Using the Gauss-Legendre rule when integrating in $\theta$ and Jacobi for $r$ we obtain optimal rates of convergence. 

%%%%%%%%%%%%%%%%%%%%%%%%%%%%%%%%%%%%%%%%%%%%
\subsubsection{Approximation of hyper-singular Integrals}
\label{sec:hypmods}
%%%%%%%%%%%%%%%%%%%%%%%%%%%%%%%%%%%%%%%%%%%%
As it was pointed out, the discretization basis for the hyper-singular operator BIO ($p^{o,r}_m$) vanishes on the boundary $\partial \Gamma$, and consequently, the entries of the corresponding matrix can be computed using the integration-by-parts formula \cite[Corollary 3.3.24]{Sauter:2011}:
\begin{align}
\label{eq:Wsplit}
 \begin{split}
\langle W_\Gamma[k] p^{o,r}_m,p^{o,r}_l \rangle  = \int_\Gamma \int_\Gamma \frac{e^{ik \|\vx -\vx' \|}}{4 \pi \| \vx -\vx'\|}curl_\Gamma p^{o,r}_m (\vx') \cdot \overline{curl_\Gamma p^{o,r}_l (\vx)} d\vx' d\vx \\-k^2
\int_\Gamma \int_\Gamma \frac{e^{ik \|\vx -\vx' \|}}{4 \pi \| \vx -\vx'\|} \widehat{\vn}(\vx) \cdot \widehat{\vn}(\vx') p^{o,r}_m (\vx')\overline{p^{o,r}_l (\vx)} d\vx' d\vx,
\end{split}
\end{align} 
where $\widehat{\vn}$ denote the unitary normal vector to $\Gamma$ (with a fixed orientation) and $curl_\Gamma f  = \widehat{\vn} \times \nabla f$, whenever $f$ can be extended to a neighborhood of $\Gamma$.
We start by considering the second integral on the right-hand side, we reduce the computations to $ \mathbb{D}$ and obtain 
\begin{align*}
k^2 
\int_\mathbb{D} \int_\mathbb{D} \frac{e^{ik \|\vr(\vx) -\vr(\vx') \|}}{4 \pi \| \vr(\vx) -\vr(\vx')\|} \widehat{\vn}(\vr(\vx)) \cdot \widehat{\vn}(\vr(\vx')) p^{o}_m (\vx')\overline{p^{o}_l (\vx)} \frac{J_\vr(\vx)}{\|\vx\|}\frac{J_\vr(\vx')}{\|\vx'\|}d\vx' d\vx,
\end{align*}
Since functions $p^{o}_l$ can be characterized as the product between a smooth function and the weight function $w_\mathbb{D}$ the same change of variables used for the weakly-singular case works for this integral\footnote{It is necessary to change the parameters of the Jacobi quadrature rule, as now the singularity is of the form $\sqrt{1-x^2}$, instead of $1/\sqrt{1-x^2}$.}. 
For the first integral in the right-hand side in \eqref{eq:Wsplit}, we compute the surface $curl$ operators. Using the parametrization of $\Gamma$, the explicit expression
$$ (curl_\Gamma f) (\vr) = \frac{1}{J_\vr }\left( \partial_{x_2} \vr \partial_{x_1} (f \circ r) - \partial_{x_1} \vr \partial_{x_2} (f \circ \vr)\right) $$
arises, where $\partial_{x_1}, \partial_{x_2}$ denote the partial derivatives with respect to the arguments of the parametrization $\vr$. In our implementation, $\vr$ is given in polar coordinates and so $x_1$ and $x_2$ are the radial and angular variables, respectively. The function $f\circ \vr $ corresponds to $p^o_m$ for some $m$,and thus, we require their partial derivatives. Moreover, using the two-indices representation we can write $p^o_m = p^\cl_\cm$, for a pair of integers $\cl,\cm$ such that $\cm+\cl$ is odd. The angular derivative is given by 
$$\partial_\theta p^{\cl}_\cm = \partial_\theta \left( C^\cl_\cm \mathrm{P}^\cl_{|\cm|}(\sqrt{1-r^2})e^{icm\theta}\right) = i\cm C^\cl_\cm  \mathrm{P}^\cl_{|\cm|} (\sqrt{1-r^2})e^{i \cm\theta} = 
i \cm p^\cl_\cm,$$
where the first equality is the explicit definition of the projected basis \eqref{eq:projectedexplicit}. We can express this in terms of the basis used for the discretization of the weakly-singular BIO as we have the following recursive relation (see \cite[15.87]{ARFKEN2013715}):
\begin{align*}
\mathrm{P}^\cl_{\cm}(x) = \frac{-\sqrt{1-x^2}}{2\cm x }\left(\mathrm{P}^\cl_{\cm+1}(x) +(\cl+\cm)(\cl-\cm+1)\mathrm{P}^\cl_{\cm-1}(x) \right),  \quad \cm \neq 0
\end{align*}
so we conclude that 
\begin{align}
\label{eq:derivativeTheta}
\partial_\theta p^\cl_\cm =  \frac{-ir}{2}\left(a^\cl_\cm q^\cl_{\cm+1}+b^\cl_\cm q^\cl_{\cm-1} \right),\end{align}
where 
\begin{align}
\label{eq:devconstants}
a^\cl_\cm := e^{-i\theta}\begin{cases}
\sqrt{(\cl-|\cm|)(\cl+|\cm|+1)}, \quad \cm \geq 0\\
-\sqrt{(\cl+|\cm|)(\cl-|\cm|+1)}, \quad   \cm < 0 \end{cases} \\
b^\cl_\cm :=  e^{i \theta}\begin{cases}
\sqrt{(\cl+|\cm|)(\cl-|\cm|+1)}, \quad  \cm \geq 0 \\
-\sqrt{(\cl-|\cm|)(\cl+|\cm|+1)}, \quad  \cm < 0 \end{cases}
\end{align}
Notice that,  for $\cm=0$, the derivative is zero and this is also true for expression \eqref{eq:derivativeTheta}. Since $\cm+\cl$ is odd, we have that $q^\cl_{\cm+1}$ and $q^\cl_{\cm-1}$ are even functions. For the derivative with respect to $r$, we need the derivative of the associated Legendre functions, given by the following recursion (see \cite[15.91]{ARFKEN2013715}):
\begin{align*}
\frac{d \mathrm{P}^\cl_{\cm}(x)}{dx} = \frac{1}{2 \sqrt{1-x^2}}\left(-\mathrm{P}^\cl_{\cm+1}(x)+
(\cl+\cm)(\cl-\cm+1)\mathrm{P}^\cl_{\cm-1}(x) \right) ,
\end{align*}
Hence, we obtain 
$$\partial_r p^\cl_\cm = \frac{1}{2}
\left(a^\cl_\cm q^\cl_{\cm+1}-b^\cl_\cm q^\cl_{\cm-1}  \right),$$
where $a^\cl_\cm, b^\cl_\cm$ are defined as in \eqref{eq:devconstants}.  Again, we have expressed the derivative in terms of the basis of the weakly-singular case. We conclude that for the computation of the hyper-singular BIO only minor modifications respect to the weakly-singular one are needed. These modifications are the change of the parameter of the respective Jacobi rule, the inclusion of the product of the normal vectors, and  the extra smooth factors $e^{\pm{\imath} \theta}, \partial_{x_j} \vr \cdot \partial_{x_j'}\vr$ ($j,j'= 1,2$) that have to be included in the kernel function. 
We have omitted the details of smooth integrals associated with the hyper-singular BIO as they do no present any extra challenge.
\subsection{Numerical Implementation}
\label{sec:NumImplementaiton}
%%%%%%%%%%%%%%%%%%%%%%%%%%%%%%%%%%%
Throughout this section we denote by $\widehat{N} = I(N)+1 = N+ \frac{N(N+1)}{2}+1$, the number of degrees of freedom when using the spaces $\mathbb{Q}^N_e(\Gamma)$, $\mathbb{P}^N_o(\Gamma)$ for the discretization of the underlying integral equations. 
Every integral needed to assembly the matrix discretization of the weakly-singular BIO ($I^a$, $I^b$, $I^c$, $I^d$, $I^e$, $I^f$ ) is a four-dimensional integral. The total number of these integral computations can be reduced using the matrix symmetries\footnote{The matrix associated with the real part of the fundamental solution ($I^\alpha, \alpha \in \{a,b,c,d,e\}$) is symmetric, while the imaginary one ($I^f$) is anti-symmetric.}. Consequently, only $\frac{\widehat{N}(\widehat{N}+1)}{2}$ combinations are needed instead of $\widehat{N}^2$.
Furthermore, since $p_{-m}^l = (-1)^m \overline{p_m^l}$ the actual number of interactions needed to compute the weakly-singular BIO is $\frac{1}{4} \left(\widehat{N} +\frac{N+2}{2} \right) \left(\frac{3}{2}\widehat{N} -\frac{N-2}{4} \right)$, assuming $N$ even\footnote{If $N$ is odd, the computational cost is $ \frac{1}{4} \left(\widehat{N} +\frac{N+1}{2} \right) \left(\frac{3}{2}\widehat{N} -\frac{N-3}{4} \right)$.}.

A four-dimensional integral computed by tensorized 1D-quadrature rules, with parameters
 $N_q^1,N_q^2$, $N_q^3,N_q^4$, has a computational cost of $\mathcal{O}(\Pi_{j=1}^4
N_q^j)$ operations and evaluations. To compute integrals arising from the weakly singular BIO $I^\alpha \in \{ b,c,d,e\}$, we denote by $N_q^\theta$ the number of points for the $\theta$ variable, and  $N_q^\alpha$ the number of points for the other three variables depending on $\alpha$. For  $\alpha=a$, we use $N_q^\theta$ for variables $\theta$ and $\beta$, and $N^a_q$ for the rest. For the hyper-singular case, same rules apply.
  
For the smooth integrals $\alpha \in \{ f,g\}$ we could use $N_q^\theta$ for the $\theta$ and $\theta'$ variables, and $N^\alpha_q, \alpha \in \{f,g\}$ for $\vr$, and $\vr'$. However, in practice it is better to reformulate the integrals onto the sphere, where the basis correspond to spherical harmonics, and approximate the integrals using the spherical harmonics transforms. In particular we use the implementation detailed in \cite{sthns}.
%%%%%%%%%%%%%%%%%%%%%%%%%%%%%%%%%%%%%%
\section{Full Discretizaton Error Analysis}
%%%%%%%%%%%%%%%%%%%%%%%%%%%%%%%%%%%%%%
The rate of convergence of the spectral Galerkin discretization method was already established in Theorem \ref{thrm:rateConvergence}. Yet, this does not illustrate the performance of the fully discrete method as extra error terms appears as consequence of the quadrature approximation of integral terms. Thus, we first measure the perturbation in error convergence rates due to quadrature error. 

\subsection{Quadrature Error}
\label{sec:quaderror}
For the sake of brevity, we denote $Q^\alpha$ the quadrature approximation of $I^\alpha$, $\alpha \in \{a,b,c,d,e,f,g\}$, defined in Section \ref{sec:AproxSingInt}. Since we assume that the screen is parametrized by a $\rho-$analytic function, and the approximation is optimal, we have that 
\begin{align}
\label{eq:quaderror0} 
| I^\alpha - Q^\alpha | \lesssim  \rho^{-N_q^\theta}+ \rho^{-N_q^\alpha}. 
\end{align}
While this bound is precise in terms of how the quadrature error decreases with increasing number of quadrature points, the unspecified constant depends on the trial and test basis indices. Hence, since the rate of convergence depends on the number of trial functions, we need a more detailed quadrature error analysis considering the exact index of the trial and test basis. For this, let us consider the canonical integral 
\begin{align*}
I^l_m = \int_0^1 \int_{-\frac{\pi}{2}}^{\frac{3\pi}{2}} g(\vx) p^l_m(\vx) r d\theta dr,
\end{align*}
where $\vx = r \mathbf{e}_\theta$,  $g$ is $\rho-$analytic in $(r,\theta)$ and $l+m$ is even. It is enough to consider this case, as all integrals discussed in Section \ref{sec:AproxSingInt} can be reduced to this form by means of analytic change of variables to tensorization of integral. Denote by $Q^l_m$ the quadrature approximation of $I^l_m$ obtained by a Gauss-Legendre rule in both variables, with $N^\theta_q$ points in the $\theta$ variable, and $N_q$ points in the $r$ variable. 

We denote by $\mathcal{E}_\rho[a,b]$ the region enclosed by the Bernstein ellipse in the complex plane with foci in $a,b$ and parameter $\rho$. Now, we recall the classical error bound for analytic integrands. 

\begin{theorem}[Theorem 5.3.13 in \cite{Sauter:2011}]
\label{thrm:quadbound}
If $m+l$ is even, for $g$ $\rho$-analytic in $[0,1]$ in the radial variable and $\rho$-analytic in $[-\frac{\pi}{2},\frac{3\pi}{2}]$ for the angular one, it holds that
\begin{align*}
\begin{split}
|I^l_m - Q^l_m | \lesssim (2\rho)^{-2N_q^\theta} \max_{r \in [0,1]} \max_{z \in \partial \mathcal{E}_\rho[\frac{-\pi}{2},\frac{3\pi}{2}]} | g(r,z) p^l_m(r,z)|   \\+
(2\rho)^{-2N_q} \max_{\theta \in [\frac{-\pi}{2},\frac{3\pi}{2}]} \max_{z \in \partial \mathcal{E}_\rho[0,1]}| g(z,\theta) p^l_m(z,\theta)| 
\end{split}
\end{align*}
where the unspecified constant does not depend on the integrand of $I^l_m$.
\end{theorem}

Since $g$ is assumed to be known, we can further simplify the error bound as 
\begin{align*}
\begin{split}
|I^l_m - Q^l_m | \lesssim (2\rho)^{-2N_q^\theta} \max_{r \in [0,1]} \max_{z \in \partial \mathcal{E}_\rho[\frac{-\pi}{2},\frac{3\pi}{2}]} | p^l_m(r,z)|   \\+
(2\rho)^{-2N_q} \max_{\theta \in [\frac{-\pi}{2},\frac{3\pi}{2}]} \max_{z \in \partial \mathcal{E}_\rho[0,1]}|  p^l_m(z,\theta)| ,
\end{split}
\end{align*}
with a constant that now depends on $g$.

\begin{corollary}
Under hypothesis of Theorem \ref{thrm:quadbound}, for the integral
\begin{align*}
\widetilde{I}^l_m :=  \int_0^1 \int_{-\frac{\pi}{2}}^{\frac{3\pi}{2}} g(\vx) q^l_m(\vx) r d\theta dr,
\end{align*}
it holds that 
\begin{align*}
\begin{split}
|\widetilde{I}^l_m - \widetilde{Q}^l_m | \lesssim (2\rho)^{-2N_q^\theta} \max_{r \in [0,1]} \max_{z \in \partial \mathcal{E}_\rho[\frac{-\pi}{2},\frac{3\pi}{2}]} | p^l_m(r,z)|   \\+
(2\rho)^{-2N_q} \max_{\theta \in [\frac{-\pi}{2},\frac{3\pi}{2}]} \max_{z \in \partial \mathcal{E}_\rho[0,1]}|  p^l_m(z,\theta)|,
\end{split}
\end{align*}
where $\widetilde{Q}^l_m$ denotes the quadrature approximation using a Gauss-Legendre rule with $N_q^\theta$ points in $\theta$, and a Jacobi rule with $N_q$ points in the $r$ variable. 
\end{corollary}

We now proceed to estimate maxima of analytic extensions for the functions $p^l_m$. 
Remember that the explicit definition \eqref{eq:projectedexplicit} (see Section \ref{sec:SH}):
\begin{align*}
p^l_m(r,\theta) = C^l_{m} \mathrm{P}^l_{|m|}(\sqrt{1-r^2})e^{im\theta},
\end{align*} 
where $|C^l_{m}| = \sqrt{\frac{(2l+1)(l-|m|)!}{2\pi (l+|m|)!}}$, with $\mathrm{P}^l_{|m|}(\sqrt{1-r^2})$ smooth in the $r$ variable. By using \cite[Theorem 3]{LOHOFER1991226}, one deduces that
\begin{align*}
\max_{r \in [0,1]} \max_{z \in \partial \mathcal{E}_\rho[\frac{-\pi}{2},\frac{3\pi}{2}]} | p^l_m(r,z)| \lesssim \sqrt{2l+1} e^{|m|\frac{\pi}{2}(\rho-\rho^{-1})} \lesssim 
\sqrt{2l+1} e^{l\pi\rho} ,
\end{align*}
where the last inequality follows for $|m| \leq l$ and $\rho >1$. On the other hand, the second term can be bounded as 
\begin{align*}
\max_{\theta \in [\frac{-\pi}{2},\frac{3\pi}{2}]} \max_{z \in \partial \mathcal{E}_\rho[0,1]}|  p^l_m(z,\theta)| \leq |C^l_{m} | 
\max_{z \in \partial \mathcal{E}_\rho[0,1]}|  \mathrm{P}^l_{|m|}(\sqrt{1-z^2}) |.
\end{align*}
We can use the Rodr\'iguez formula to express the associated Legendre function in terms of Legendre polynomials:
\begin{align*}
|\mathrm{P}^l_{|m|}(\sqrt{1-z^2})| = |z^m|  \left\vert\left(\frac{d^m}{dx^m}\mathrm{P}^l(x)\right) |_{x = \sqrt{1-z^2}} \right\vert,
\end{align*}
where $\mathrm{P}^l$ denotes the $l$th Legendre polynomial. Obviously, $|z|  < \frac{\rho+\rho^{-1}}{4}$, for every $z \in \partial \mathcal{E}_\rho[0,1]$. Moreover,
\begin{align*}
\max_{z \in \partial \mathcal{E}_\rho[0,1]} \left\vert\left(\frac{d^m}{dx^m} \mathrm{P}^l(x)\right) \Big|_{x = \sqrt{1-z^2}} \right\vert = \max_{z \in A_\rho}  \left\vert\left(\frac{d^m}{dx^m} \mathrm{P}^l(x)\right) \Big|_{x = z}\right\vert,
\end{align*}
where $A_{\rho}$ is the image of $\partial \mathcal{E}_\rho[0,1]$ under the transformation $(1-x^2)^{\half}$, where the square root has to be understood as the pre-image of the square function. Since $\mathrm{P}^l$ are polynomials and using the maximum modulus principle, there exists $\widehat{\rho} >1$ such that 
\begin{align*}
\max_{z \in A_\rho}  \left\vert\left(\frac{d^m}{dx^m} \mathrm{P}^l(x)\right) \Big|_{x = z}\right\vert \leq \max_{z \in \partial \mathcal{E}_{\widehat{\rho}}[-1,1]}  \left\vert\left(\frac{d^m}{dx^m} \mathrm{P}^l(x)\right) \Big|_{x = z}\right\vert.
\end{align*}
Furthermore, using the Cauchy integral formula we have that
\begin{align*}
\frac{d^m}{dz^m} \mathrm{P}^l(z) = \frac{m!}{2 i\pi } \int_{\partial \mathcal{E}_{2\widehat{\rho}}[-1,1]} \frac{\mathrm{P}^l(x)}{(z-x)^{m+1}}dx, \quad \forall z \in \partial \mathcal{E}_{\widehat{\rho}}[-1,1].
\end{align*}
Hence, by using \cite{Wang2018}[Theorem 4.1] we have 
\begin{align*}
\max_{z \in \partial \mathcal{E}_{\widehat{\rho}}[-1,1]}  \left\vert\left(\frac{d^m}{dx^m} \mathrm{P}^l(x)\right) |_{x = z}\right\vert \leq \frac{m!}{2 \pi } L(\partial \mathcal{E}_{2\widehat{\rho}}[-1,1]) \mathrm{P}^l(\widehat{\rho} + (4 \widehat{\rho})^{-1})\max_{x \in \partial \mathcal{E}_{2\widehat{\rho}}[-1,1]} \frac{1}{|z-x|^{m+1}}, 
\end{align*}
 where $L(\partial \mathcal{E}_{2\widehat{\rho}}[-1,1])$ is the length of the corresponding ellipse, and as such, it can be estimated as $L(\partial \mathcal{E}_{2\widehat{\rho}}[-1,1]) \lesssim \widehat{\rho}$. Also, notice that the minimum distance between  $\partial \mathcal{E}_{2\widehat{\rho}}[-1,1]$ and $\partial \mathcal{E}_{\widehat{\rho}}[-1,1]$ is larger\footnote{This can be shown using elementary geometrical computations.} than $\widehat{\rho}$. Thus, one has
\begin{align*}
\max_{z \in \partial \mathcal{E}_{\widehat{\rho}}[-1,1]}  \left\vert\left(\frac{d^m}{dx^m} \mathrm{P}^l(x)\right) |_{x = z}\right\vert \lesssim 
\frac{m!}{\widehat{\rho}^{m}} \mathrm{P}^l(\widehat{\rho} + (4 \widehat{\rho})^{-1}) \lesssim \frac{m!}{2^l}\widehat{\rho}^{l-m},
\end{align*}
wherein the rightmost inequality follows for $l$ is large as the polynomial is dominated by the monomial of greatest degree. Finally, combining all the computations for $l$ large enough, we have that
\begin{align*}
\max_{\theta \in [\frac{-\pi}{2},\frac{3\pi}{2}]} \max_{z \in \partial \mathcal{E}_\rho[0,1]}|  p^l_m(z,\theta)| \lesssim |C^l_{|m|} | 
m! \widehat{\rho}^{l}  \lesssim \sqrt{\frac{(2l+1)(l!)^2}{(2l)!}} \widehat{\rho}^l  \lesssim \sqrt{2l+1} \widehat{\rho}^l 
\end{align*}
and the quadrature error is then bounded as 
\begin{align*}
|{I}^l_m - {Q}^l_m | \lesssim \sqrt{2l+1} \left[ (2\rho)^{-2N_q^\theta}  e^{\pi l \rho}   + (2\rho)^{-2N_q} \widehat{\rho}^l\right].
\end{align*}
This bound be further simplified to
\begin{align*}
|{I}^l_m - {Q}^l_m |  \lesssim \sqrt{2l+1} \left[ (2\rho)^{-2N_q^\theta}     + (2\rho)^{-2N_q} \right] \widetilde{\rho}^l 
\end{align*}
where $\widetilde{\rho}>1$.
We can summarize quadrature error bound in the following result: 
\begin{lemma}
\label{lemma:quaderror}
There is $N_0 \in \mathbb{N}$ such that given integers $m,l$, $|m| <l $ and $l>N_0$, and $g :[0,1] \times [\frac{-\pi}{2},\frac{3\pi}{2}] \rightarrow \mathbb{C}$ $\rho-$analytic in both variables. For the integral
$$I^l_m := \int_0^1\int_{\frac{-\pi}{2}}^{\frac{3\pi}{2}} g(\vx) p^l_m(\vx)r d\theta dr,$$
we have the error bound 
\begin{align*}
|{I}^l_m - {Q}^l_m |  \lesssim \sqrt{2l+1} \left( (2\rho)^{-2N_q^\theta}     + (2\rho)^{-2N_q} \right) \widetilde{\rho}^l 
\end{align*}
where $Q^l_m$ denotes the approximation by Gauss-Legendre of order $N_q,N_q^\theta$, for both variables accordingly, $\widetilde{\rho}>1$, and the unspecified constant does not depend on $l,m$.
\end{lemma} 

\begin{corollary}
\label{cor:quadtenzerror}
Under the hypothesis and notations of Lemma \ref{lemma:quaderror}, given another pair of integers $l',m'$ such that $|m'| <l' $ and $l'>N_0$, and $G(\vx,\vx') : \left([0,1] \times [\frac{-\pi}{2},\frac{3\pi}{2}] \right)^2 \rightarrow \mathbb{C}$, for the integral
\begin{align*}
I^{l,l'}_{m,m'} :=  \int_0^1 \int_{-\frac{\pi}{2}}^{\frac{3\pi}{2}} \int_0^1 \int_{-\frac{\pi}{2}}^{\frac{3\pi}{2}} G(\vx,\vx') p^l_m(\vx) p^{l'}_{m'}(\vx') r r' d\theta dr d\theta' dr',
\end{align*}
it holds that
\begin{align*}
|I^{l,l'}_{m,m'} - Q^{l,l'}_{m,m'} |  \lesssim 
\sqrt{(2l+1)(2l'+1)}\left( (2\rho)^{-2N_q^\theta}     + (2\rho)^{-2N_q} \right) \widetilde{\rho}^l\widetilde{\rho}'^{l'}
\end{align*}
where $Q^{l,l'}_{m,m'}$ denotes the Gauss-Legendre quadrature of orders $N_q$ in $r,r'$ and $N_q^\theta$ for $\theta,\theta'$, $\widetilde{\rho}> 1$ and $\widetilde{\rho}' > 1$, and the unspecified constant does not depend of $l,m,l',m'$.
\end{corollary}

%%%%%%%%%%%%%%%%%%%%%%%%%%%%
\subsection{Fully Discrete Error Analysis:}
%%%%%%%%%%%%%%%%%%%%%%%%%%%%
Recall Problem \ref{prob:discreteBIEs}, where the unknowns are vectors of dimension $I(N) = \frac{N(N+1)}{2}+N$. Let us  now consider the same problem where the corresponding matrices and righ-hand sides are approximated with the quadrature method detailed in Section \ref{sec:MatrixComputations}.

\begin{problem}
\label{prob:FullDiscrete}
Find $\vlambda^{N,q},\vnu^{N,q}  \in \mathbb{C}^{I(N)+1}$ such that 
\begin{align}
V^{N,q}[k] \vlambda^{N,q} = \vg_d^{N,q}, \quad \text{(Fully discrete Dirichlet BIE)}, \\
W^{N,q}[k] \vnu^{N,q} = \vg_n^{N,q}, \quad \text{(Fully discrete Neumann BIE)},
\end{align}
where $V^{N,q}[k]$ (resp.~$W^{N,q}[k],\vg_d^{N,q}, \vg_n^{N,q}$) is the quadrature approximation to $V^N[k]$ (resp.~$W^{N}[k],\vg_d^{N}, \vg_n^{N}$) constructed as described in Section \ref{sec:MatrixComputations}. 
\end{problem}
The approximations obtained from the fully discrete problems are written
\begin{align*}
\lambda^{N,q}_e = \sum_{m=0}^{I(N)} \lambda^{N,q}_m q_m^{e,r},\quad
\nu^{N,q}_o = \sum_{m=0}^{I(N)} \nu^{N,q}_m p_m^{o,r}.
\end{align*}
We refer to Problem \ref{prob:FullDiscrete} as fully discrete problems. For their analysis, we detail the Dirichlet case as the Neumann case follows similar ideas with results for both is reported at the end of the section.

We first estimate the quadrature error between matrices $V^N[k]-V^{N,q}[k]$.  Recall the notation introduced at beginning of Section \ref{sec:DiscreteProblems}. For simplicity, the number quadrature points is determined by only two variables $N_q^\theta$ and $N_q$. Following the notation in Section \ref{sec:NumImplementaiton}, this corresponds to the simpler case $N_q^\alpha = N_q$ for every $\alpha \in \{a,b,c,d,e,f\}$. For different values of $N_q^\alpha$, the results are essentially the similar but we forgo this analysis for the sake of brevity.
\begin{lemma}
\label{lemma:quadError}
There is $N_0 \in \mathbb{N}$, such that for $N>N_0$, given vectors $\boldsymbol{\mu}^N$, $\boldsymbol{\psi}^N\in \mathbb{C}^{I(N)+1}$, we have that 
\begin{align*}
\begin{split}
\left|\boldsymbol{\psi}^N \cdot \left(V^N[k]-V^{N,q}[k] \right) \boldsymbol{\mu}^N \right| \lesssim \widetilde{\rho}^{2 N}N^4 \left[(2\rho)^{-2N_q}+(2\rho)^{-2N^\theta_q}\right]\\
\times \| \mu_e^N \|_{\widetilde{H}^{-\half}(\Gamma)} \| \psi_e^N \|_{\widetilde{H}^{-\half}(\Gamma)}, 
\end{split}
\end{align*}
where, as before, $\rho>1,\widetilde{\rho} >1$.
\end{lemma}
\begin{proof}
Expanding matrix products yields
\begin{align*}
\begin{split}
&\left|\boldsymbol{\psi}^N \cdot \left(V^N[k]-V^{N,q}[k] \right) \boldsymbol{\mu}^N \right|
= \\
&\left\vert \sum_{l'=0}^N \sum_{\substack{m'=-l'\\ m'+l' \text{ even}}}^{l'} \overline{ \psi^N_{I_e(l',m')}} 
\sum_{l=0}^N \sum_{\substack{m=-l\\ m+l \text{ even}}}^{l} \left(V^N[k]-V^{N,q}[k] \right)_{I_e(l',m'),I_e(l,m)} \mu^N_{I_e(l,m)}
\right\vert,
\end{split} 
\end{align*}
where the index $I_e(l,m) = \frac{l(l+1)+(l+m)}{2}$ was defined in Section \ref{sec:SH}. Now, $V^N[k]$ (see Section \ref{sec:DiscreteProblems}) is a matrix whose entries can be written as the sum of integrals $I^a,I^b,I^c,I^d,I^e,I^f$, and $V^{N,q}[k]$ is the sum corresponding to quadratures approximations  $Q^a,Q^b, Q^c,Q^d, Q^e,Q^f$ as described in Section \ref{sec:AproxSingInt}. By construction, we can use\footnote{Each of the integrals $I^\alpha, \alpha \in \{ a,b,c,d,e,f\}$ have different integration domain but they could be fixed using smooth window functions that extend the integrands to the required domain so the application of the Corollary directly. } Corollary \ref{cor:quadtenzerror} and obtain  
\begin{align*}
\begin{split}
&\left|\boldsymbol{\psi}^N \cdot \left(V^N[k]-V^{N,q}[k] \right) \boldsymbol{\mu}^N \right|
\lesssim \left[(2\rho)^{-2N_q^\theta}+(2\rho)^{-2N_q}\right]\\&{\times} \left\vert\sum_{l'=0}^N \sum_{\substack{m'=-l'\\ m'+l' \text{ even}}}^{l'} \sum_{l=0}^N \sum_{\substack{m=-l\\ m+l \text{ even}}}^{l}  \sqrt{(2l+1)(2l'+1)}
 \left(\widetilde{\rho}^{'l'} \widetilde{\rho}^l  \right)  \overline{\psi^N_{I_e(l',m')}}  \mu^N_{I_e(l,m)}  \right\vert
\end{split}.
\end{align*}
Redefining $\widetilde{\rho}:= \max \{\widetilde{\rho}', \widetilde{\rho} \}$  leads to
\begin{align*}
\begin{split}
&|\boldsymbol{\psi}^N \cdot \left(V^N[k]-V^{N,q}[k] \right) \boldsymbol{\mu}^N |
\lesssim  \widetilde{\rho}^{2N} \left[(2\rho)^{-2N_q^\theta}+(2\rho)^{-2N_q}) \right]\\
&\times\left\vert \sum_{l'=0}^N \sum_{\substack{m'=-l'\\ m'+l' \text{ even}}}^{l'} \sum_{l=0}^N \sum_{\substack{m=-l\\ m+l \text{ even}}}^{l} \sqrt{(2l+1)(2l'+1)} \overline{\psi^N_{I_e(l',m')} } \mu^N_{I_e(l,m)}   \right\vert.
\end{split}
\end{align*}
Applying the Cauchy-Schwartz inequality one obtains
\begin{align*}
\begin{split}
&|\boldsymbol{\psi}^N \cdot \left(V^N[k]-V^{N,q}[k] \right) \boldsymbol{\mu}^N |
\lesssim   \widetilde{\rho}^{2N} \left[(2\rho)^{-2N_q^\theta}+(2\rho)^{-2N_q}) \right]\\
&\times N^4
\left(\sum_{l'=0}^N \sum_{\substack{m'=-l'\\ m'+l' \text{ even}}}^{l'} \frac{(\psi^N_{I_e(l',m')})^2}{l'+1}\right)^{\half} \left( \sum_{l=0}^N \sum_{\substack{m=-l\\ m+l \text{ even}}}^{l}  \frac{(\mu^N_{I_e(l,m)})^2}{l+1}   \right)^{\half}
\end{split},
\end{align*}
and by the auxiliary spaces norms definitions, it holds that
\begin{align*}
\begin{split}
|\boldsymbol{\psi}^N \cdot \left(V^N[k]-V^{N,q}[k] \right) \boldsymbol{\mu}^N |
&\lesssim   \widetilde{\rho}^{2N} \left[(2\rho)^{-2N_q^\theta}+(2\rho)^{-2N_q}) \right]\\
&\times N^4
\| \psi^N \circ \vr \|_{Q_e^{-\half}(\mathbb{D})} \| \mu^N \circ \vr \|_{Q_e^{-\half}(\mathbb{D})}.
\end{split}
\end{align*}
The results follows directly from Lemmas \ref{lemma:sobRev} and \ref{lemma:SobolevLocalization}.
\end{proof}

We notice that for any fixed value of $N$ the error term 
\begin{align*}
\widetilde{\rho}^{2N} \left[(2\rho)^{-2N_q^\theta}+(2\rho)^{-2N_q}) \right] N^4
\end{align*}
goes to zero as the quadrature order increases. Thus, we can use the standard Strang's Lemma to bound the fully discrete error. 

\begin{theorem}[Theorem 4.2.11 in \cite{Sauter:2011}]
\label{theo:dirichconverg}
There exists $N_0>0$ such that for every $N>N_0$, there is a $N_{q,0}$ that depends on $N$ such that, if $N_q^\theta$, $N_q$ are both greater than $N_{q,0}$, the fully discrete Dirichlet problem \ref{prob:FullDiscrete} has a unique solution, and the following error estimate follows for $s > -\fourth$:
\begin{align}
\label{eq:direstimation}
\begin{split}
\| \lambda - \lambda_e^{N,q} \|_{\widetilde{H}^{-\half}(\Gamma)} &\lesssim N^{-\fourth-s}\| \lambda \circ \vr \|_{Q^s_e(\mathbb{D})}\\ &+\widetilde{\rho}^{2N} \left[(2\rho)^{-2N_q^\theta}+(2\rho)^{-2N_q}) \right] N^4\|\lambda\|_{\widetilde{H}^{-\half}(\Gamma)}\\
&+\widetilde{\rho}^{N} \left[(2\rho)^{-2N_q^\theta}+(2\rho)^{-2N_q}) \right] N^2,
\end{split}
\end{align}
where $\lambda$ denotes the continuous solution of the Dirichlet BIE \eqref{eq:bies}.
\end{theorem}
\begin{proof} 
Existence and uniqueness are obtained following \cite[Theorem 4.2.11]{Sauter:2011}. Moreover from the same reference, it holds that
\begin{align*}
\begin{split}
\| \lambda - \lambda^{N,q}_e \|_{\widetilde{H}^{-\half}(\Gamma)} \lesssim \inf_{v^N \in \mathbb{Q}^e_N(\Gamma)} \left( \|\lambda - v^N\|_{\widetilde{H}^{-\half}(\Gamma)} +c(N_q,N) \|v^N \|_{\widetilde{H}^{-\half}(\Gamma)} \right)\\+\sup_{v^N \in \mathbb{Q}^e_N(\Gamma) } \frac{\boldsymbol{v}^N \cdot (\vg_d^N- \vg_d^{N,q})}{\| v^N \|_{\widetilde{H}^{-\half}(\Gamma)}},
\end{split}
\end{align*}
where $c(N_q,N) {:=} \widetilde{\rho}^{2N} \left[(2\rho)^{-2N_q^\theta}+(2\rho)^{-2N_q}) \right] N^4$. The second term on the right-hand side can be estimated following the same ideas of Lemma \ref{lemma:quadError}, and we obtain:
\begin{align*}
\sup_{v^N \in \mathbb{Q}^e_N(\Gamma) } \frac{\boldsymbol{v}^N \cdot (\vg_d^N- \vg_d^{N,q})}{\| v^N \|_{\widetilde{H}^{-\half}(\Gamma)}} \lesssim \widetilde{\rho}^N \left[(2\rho)^{-2N_q^\theta}+(2\rho)^{-2N_q}) \right] N^2.
\end{align*}
With this, the result is obtained following Theorem \ref{thrm:rateConvergence}.
\end{proof}

\begin{remark}
In the right-hand side of \eqref{eq:direstimation}, the last term grows as $\widetilde{\rho}^N N^2$ due a to single numeric quadrature whereas the second term does at rate of $\widetilde{\rho}^{2N} N^4$ as it arises from the tensorization of two quadrature rules.
\end{remark}

In a similar fashion, we obtain the equivalent result for the Neumann problem. 
\begin{theorem}
\label{theo:neumannconverg}
There exists $N_0>0$ such that for every $N>N_0$, there is a $N_{q,0}$ that depends on $N$ such that, if $N_q^\theta$, $N_q$ is both greater than $N_{q,0}$, the fully discrete Neumann Problem \ref{prob:FullDiscrete} has a unique solution, and the following error estimate follows for $s > \half$: 
\begin{align}
\label{eq:neuestimation}
\begin{split}
\| \nu - \nu^{N,q}_o \|_{\widetilde{H}^{\half}(\Gamma)} \lesssim & \ N^{\half-s}\| \nu \circ \vr \|_{P^s_o(\mathbb{D})}+\widetilde{\rho}^{2N} \left[(2\rho)^{-2N_q^\theta}+(2\rho)^{-2N_q}) \right] N^2\|\nu\|_{\widetilde{H}^{\half}(\Gamma)}\\ 
&+\widetilde{\rho}^{N} \left[(2\rho)^{-2N_q^\theta}+(2\rho)^{-2N_q}) \right] N.
\end{split}
\end{align}
\end{theorem}

\begin{remark}
The differences between the second (resp.~third) term in the right hand-sides of \eqref{eq:direstimation} and \eqref{eq:neuestimation} is caused for the norm used to measure the error in each case. 
\end{remark}

%\begin{remark}\todo{write as theorem/no proof}
%A similar error bound can be obtained for the fully discrete Neumann problem, 
%\begin{align*}
%\begin{split}
%&\| \nu - \nu^{N,q} \|_{\widetilde{H}^{\half}(\Gamma)} \lesssim N^{\half-s}\| \nu \circ \vr \|_{P^s_o(\mathbb{D})} +\\ &\widetilde{\rho}^{2N} \left((2\rho)^{-2N_q^\theta}+(2\rho)^{-2N_q}) \right) (N^2I(N))\|\nu\|_{\widetilde{H}^{\half}(\Gamma)}+\\
%&\widetilde{\rho}^N \left((2\rho)^{-2N_q^\theta}+(2\rho)^{-2N_q}) \right) (N\sqrt{I(N)})
%\end{split},
%\end{align*}
%for any $s > \half$.
%\end{remark}
\begin{remark}
\label{rem:linearNq}
From the fully discrete error analysis, one concludes that the number of quadrature points should be a linear function of the parameter $N$ so as to obtain good approximations of the BIOs. 
\end{remark}

%%%%%%%%%%%%%%%%%%%%%%%%%%%%%%%%%%%%%%%%
\section{Numerical Results}
\label{sec:numres3d}
%%%%%%%%%%%%%%%%%%%%%%%%%%%%%%%%%%%%%%%%
In what follows, we conduct a series of numerical experiments to verify our claims, showcase insights and show limitations of the provided results. These computational results were carried out on a desktop PC I7-4790k with 8Gb of RAM.
%%%%%%%%%%%%%%%%%%%%%%%%%%%%%%%%%%%%%%%%
\subsection{Quadrature Results}
%%%%%%%%%%%%%%%%%%%%%%%%%%%%%%%%%%%%%%%%
\begin{figure}
  \subfigure[
	$I^a,I^b,I^c$ 
  ]{   
    \includegraphics[width=.47\textwidth]{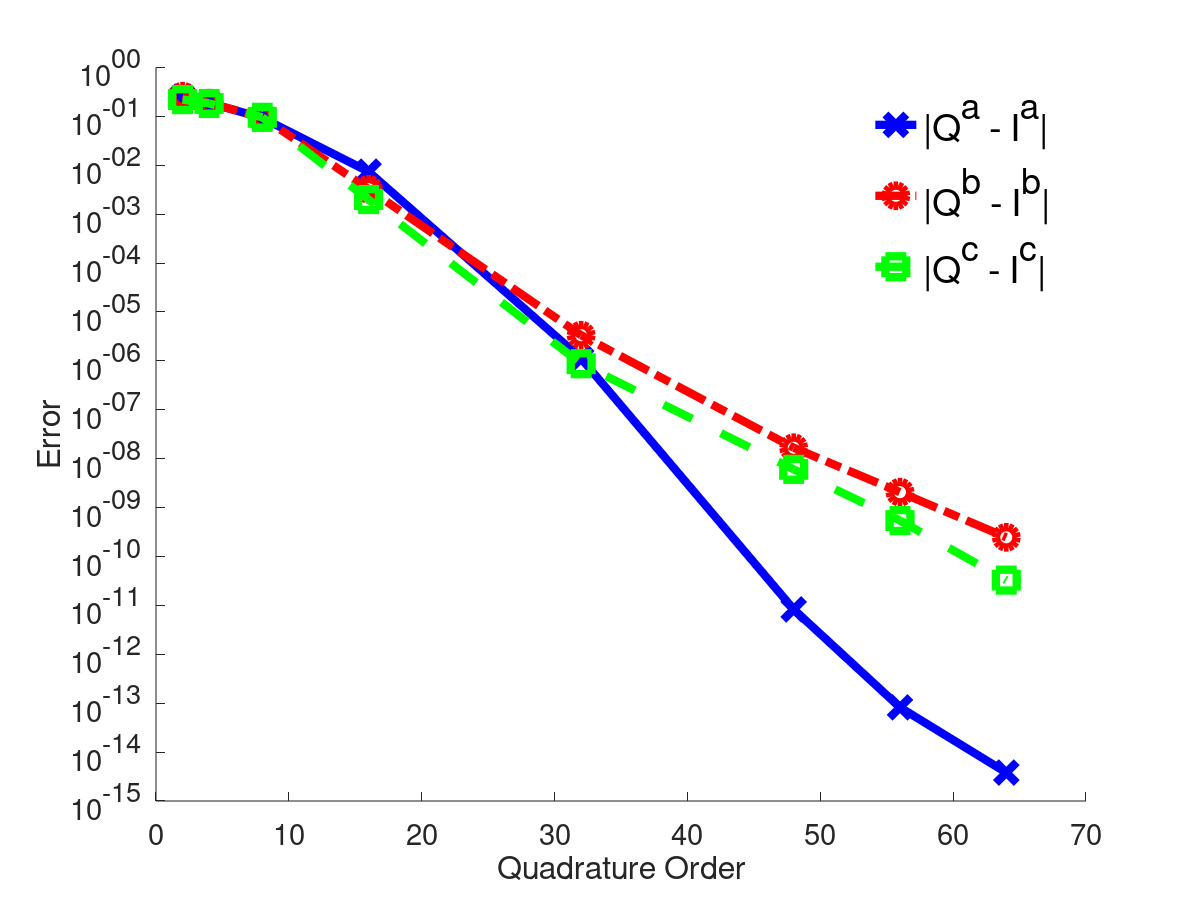}
  }
  \subfigure[
    $I^d,I^e$
  ]{
  \includegraphics[width=.47\textwidth]{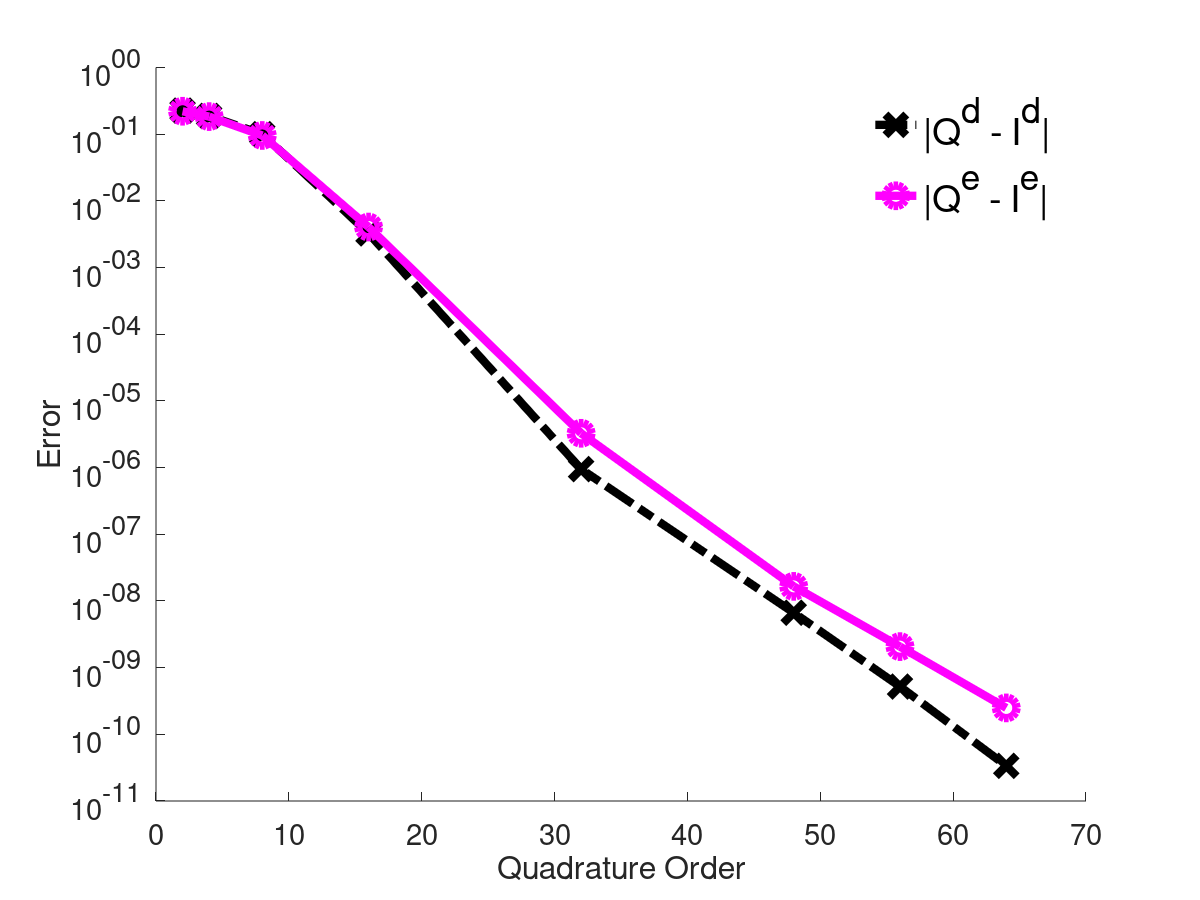}
  }
  \caption{ 
Quadrature errors for $k = 2.8$, computed against an overkill with $N_q = 76$. The error is
computed by taking the 2-norm ($\|A\|_2 =  \sup_{\vv \neq 0}\frac{\vv \cdot A \vv}{\|\vv\|}$)  we fix $N = 8$, (45 degrees of freedom).
}
\label{fig:quad1}
\end{figure} 
Before studying the accuracy of the full spectral Galerkin method, we consider the performance of the quadrature procedure detailed in Section \ref{sec:AproxSingInt}. To this end, we consider the screen $\Gamma$ given by the following parametrization: 
$$\vr(r,\theta) = r( \cos  \theta, 2.8\sin\theta,-0.56 r).$$
We compute the integrals $I^\alpha$, $\alpha \in \{a,b,c,d,e\}$ for an increasing number of quadrature points. In particular, we only select the variable $N_q^\alpha$ and fix $N_q^\theta = N_q^\alpha+ 12$. Results reported in Figure \ref{fig:quad1} show that quadrature errors decay linearly in the log-linear scale --exponential decay-- as described in \eqref{eq:quaderror0}.

As a second test, we consider the same screen and wavenumber for an increasing number of functions and quadrature points. As before, we only modify the variable $N_q^\alpha$, $\alpha \in{a,b,c,d,e,f}$ and fix $N_q^\theta = N_q^\alpha +12$. The results are presented in Table \ref{tab:errorA}. In contrast to the previous experiment, we show the error for the consolidated variable $V^N[k] - V^{N,q}[k]$ and observe that the quadrature error is almost constant when the quadrature points increase linearly with $N$ as stated in Remark \ref{rem:linearNq}.

\begin{table}
\centering
\caption{Quadrature error for the weakly-singular operator with $k = 2.8$ computed against an overkill with $N_q = 92, N=24$. The rule for increasing the quadrature points with $N$ is  $N_q(N) = 1.75 N +50$. The error again is computed as the 2-norm of the approximation of the weakly-singular operator matrix.}
\begin{tabular}{l|llllll}
$N$    & 2        & 6         & 14       & 16       & 18  & 20       \\ 
\hline
$N_q$ & 18       & 25       & 39       & 42       & 46    & 49    \\ 
\hline
Error    & 6.20e-15 & 4.89e-15  & 8.88e-15 & 9.60e-15 & 7.14e-15 & 8.39e-15
\end{tabular}
\label{tab:errorA}
\end{table}

%\begin{figure}
%\centering
%\includegraphics[scale=0.2]{quad3.png}
%\label{fig:errorA}
%\caption{Quadrature error for the weakly-singular operator $k = 2.8$ computed against an overkill with $N_q = 92, N=24$. The rule for increasing the quadrature points with $N$ is  $N_q(N) = 1.75 N +50$. The error again is computed as the 2-norm of the approximation of the weakly-singular operator matrix.}
%\end{figure}
   
%%%%%%%%%%%%%%%%%%%%%%%%%%%%%%%%%%%%%%%%
\subsection{Code Validation}
%%%%%%%%%%%%%%%%%%%%%%%%%%%%%%%%%%%%%%%%
We now show that our method is correctly implemented for the case of the Laplace Dirichlet and Neumann problem for the disk $\mathbb{D}$. Recall the closed form of the matrix entries in \eqref{eq:OpEigens}. We consider $g_d$ (resp.~$g_n$) as the Dirichlet (resp.~Neumann) traces of a plane wave: 
\begin{align*}
\exp( \imath k_0 \vx \cdot  \mathbf{d} )
\end{align*}
where $\mathbf{d}$ is an unitary vector which can be characterized in terms of two angles:
$$
\mathbf{d} = (\cos\theta_0 \cos\varphi_0, \sin\theta_0\cos\varphi_0,\sin\varphi_0),
$$
and $k_0$ is the wavenumber of the plane wave. In Figure \ref{fig:errorA} we show the error of the approximated solution with respect to the semi-analytic solution --the one obtained using  \eqref{eq:OpEigens} for the matrix computations and a fixed value of $N$. Notice that the super-algebraic convergence rate is achieved, i.e.~linear behavior in log-linear scale.

\begin{figure}
  \subfigure[
	$\Gamma_1$ 
  ]{   
   \includegraphics[scale=0.19]{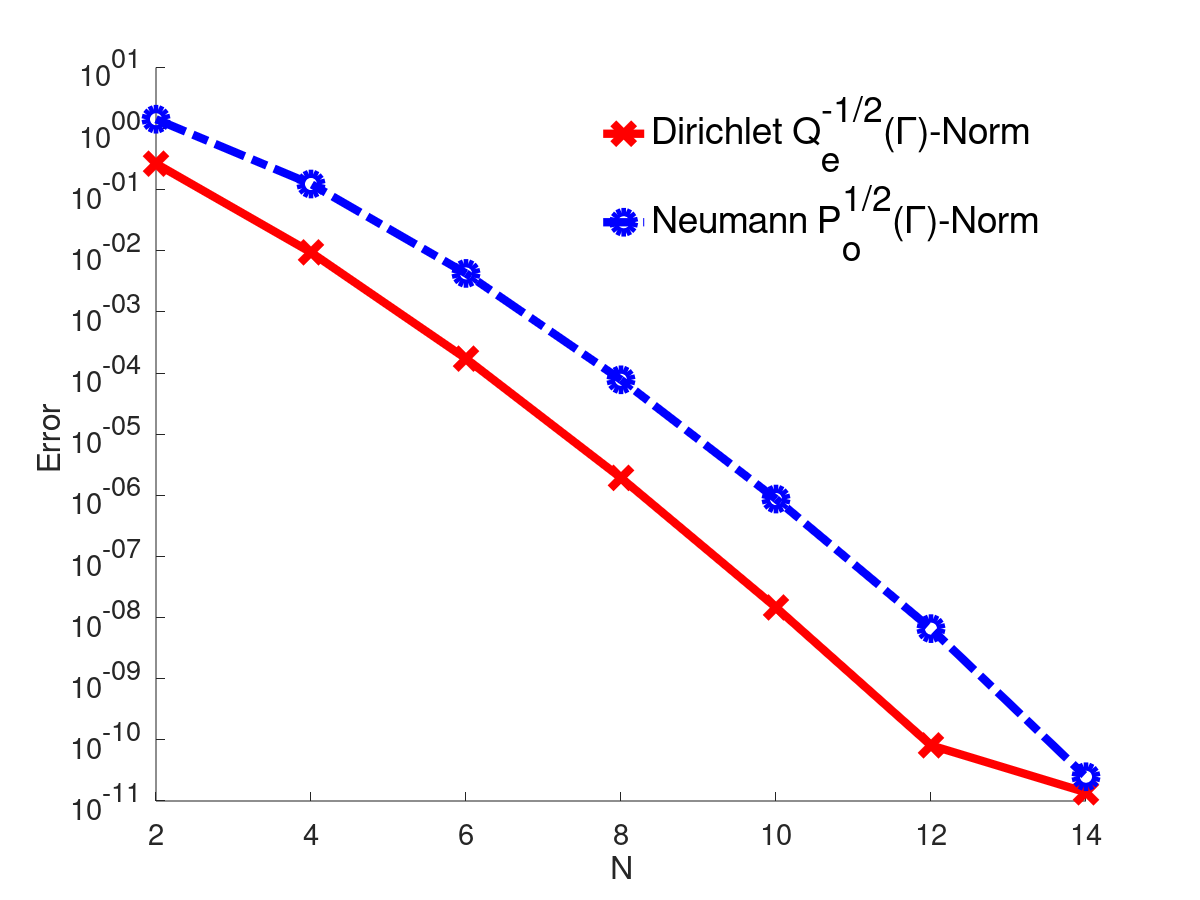}
   }
  \subfigure[
    right hand-side
  ]{
  \includegraphics[width=.47\textwidth]{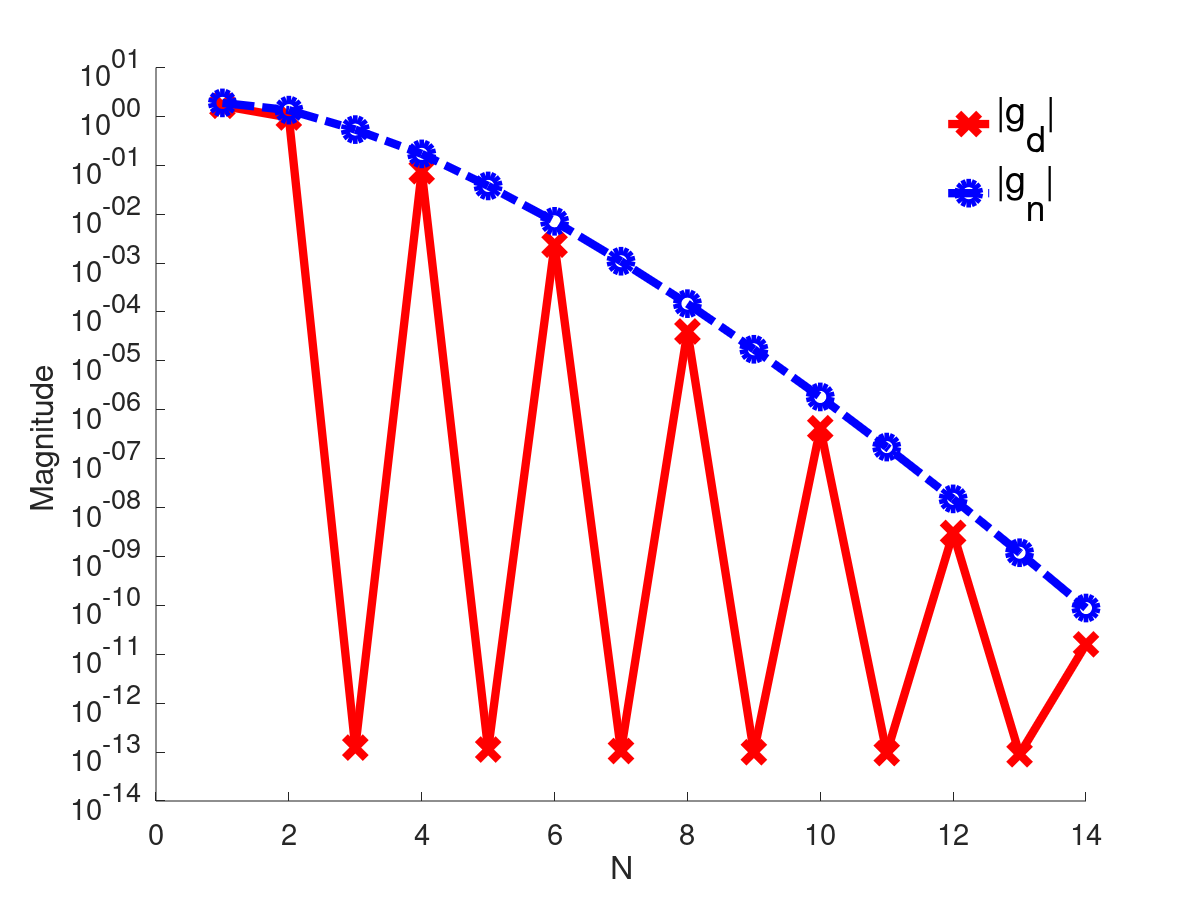}
  }
\label{fig:errorA}
\caption{(a) Convergence curves for the Laplace problems on the disk. Parameters for the plane wave are $k_0 = 2.8$, $\theta_0 = \pi/3$, $\varphi_0 = \frac{\pi}{4}$. Running times are for $N = 14$ are 181s, 730s for Dirichlet and Neumann cases, respectively.
(b) Maximum value of absolute values in the right-hand side for the corresponding $N$. }
\end{figure}
As in the two-dimensional case \cite{ElPinto}, Figure \ref{fig:errorA} suggests that one could deduce the decaying behavior of solution coefficients from the right-hand side coefficients. For the Laplace case on a disk, this comes directly from \eqref{eq:OpEigens} but for more general cases, to the best of our knowledge, this has not been done. One could prove results in this context by establishing a complete theory of pseudo-differential operators on screens acting in the auxiliary spaces, in a similar fashion \cite{Averseng2019} presented for the two-dimensional case.

 %%%% %%%% %%%% %%%% %%%% %%%%
\subsection{More complex screens}
 %%%% %%%% %%%% %%%% %%%% %%%%
We fix the wavenumbers $k = 2.8$, and  for right-hand side we use a plane wave with $k_0 = k$, $\theta_0 = \frac{\pi}{3}$, $\varphi_0 = \frac{\pi}{4}$. Let us first consider two distinct geometries: 
$\Gamma_1$  truncated elliptic screen given by 
$$\vr_1(r,\theta) = r( \cos  \theta, 2.8\sin\theta,0)$$
and  $\Gamma_2$ a truncated paraboloid, 
$$\vr_2(r,\theta) = r( \cos  \theta, 2.8\sin\theta,-0.56 r).$$

Results for $\Gamma_1, \Gamma_2 $ are presented in  Figure \ref{fig:Convg1}, wherein we obtain super-algebraic convergence --linear convergence in log-linear scale-- as stated in Theorems \ref{theo:dirichconverg}, \ref{theo:neumannconverg}. 
\begin{figure}
  \subfigure[
	$\Gamma_1$ 
  ]{   
    \includegraphics[width=.47\textwidth]{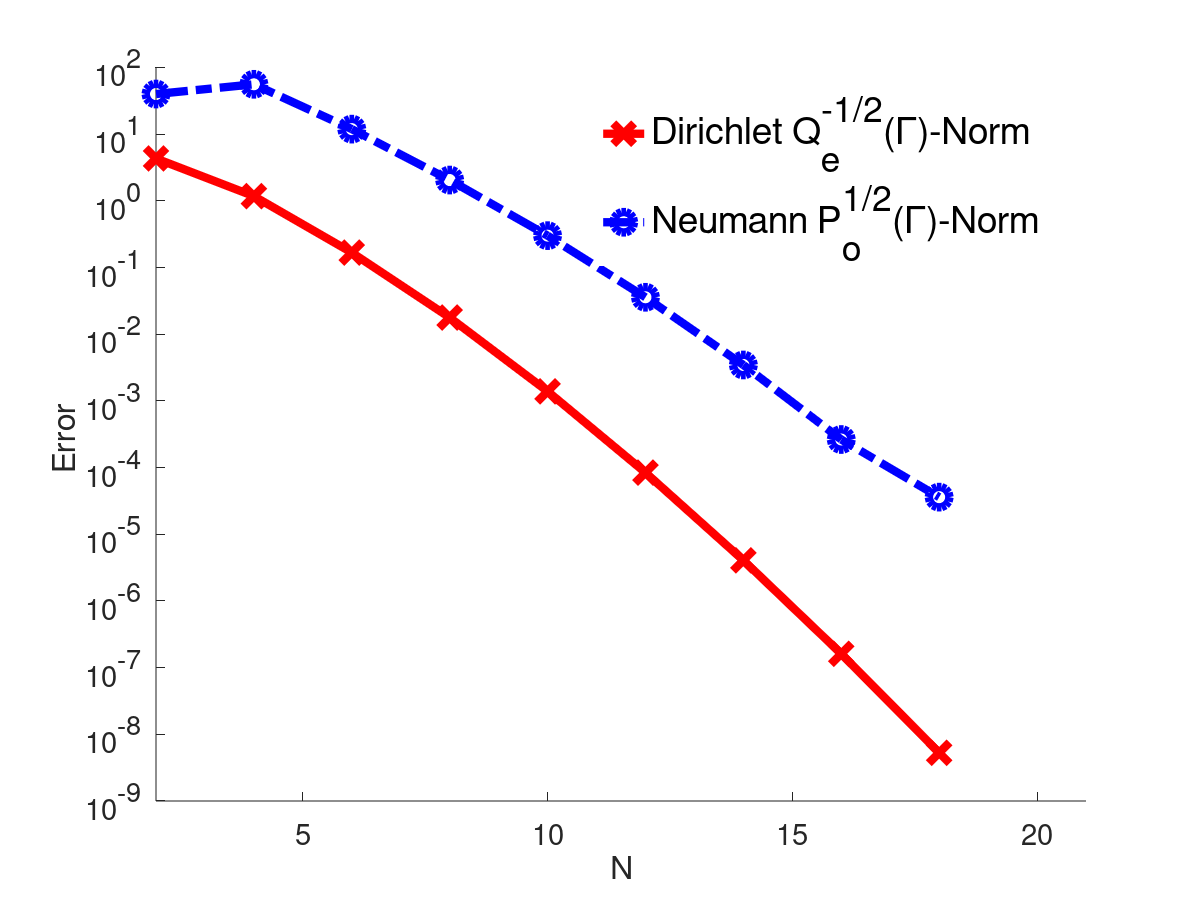}
  }
  \subfigure[
    $\Gamma_2$
  ]{
  \includegraphics[width=.47\textwidth]{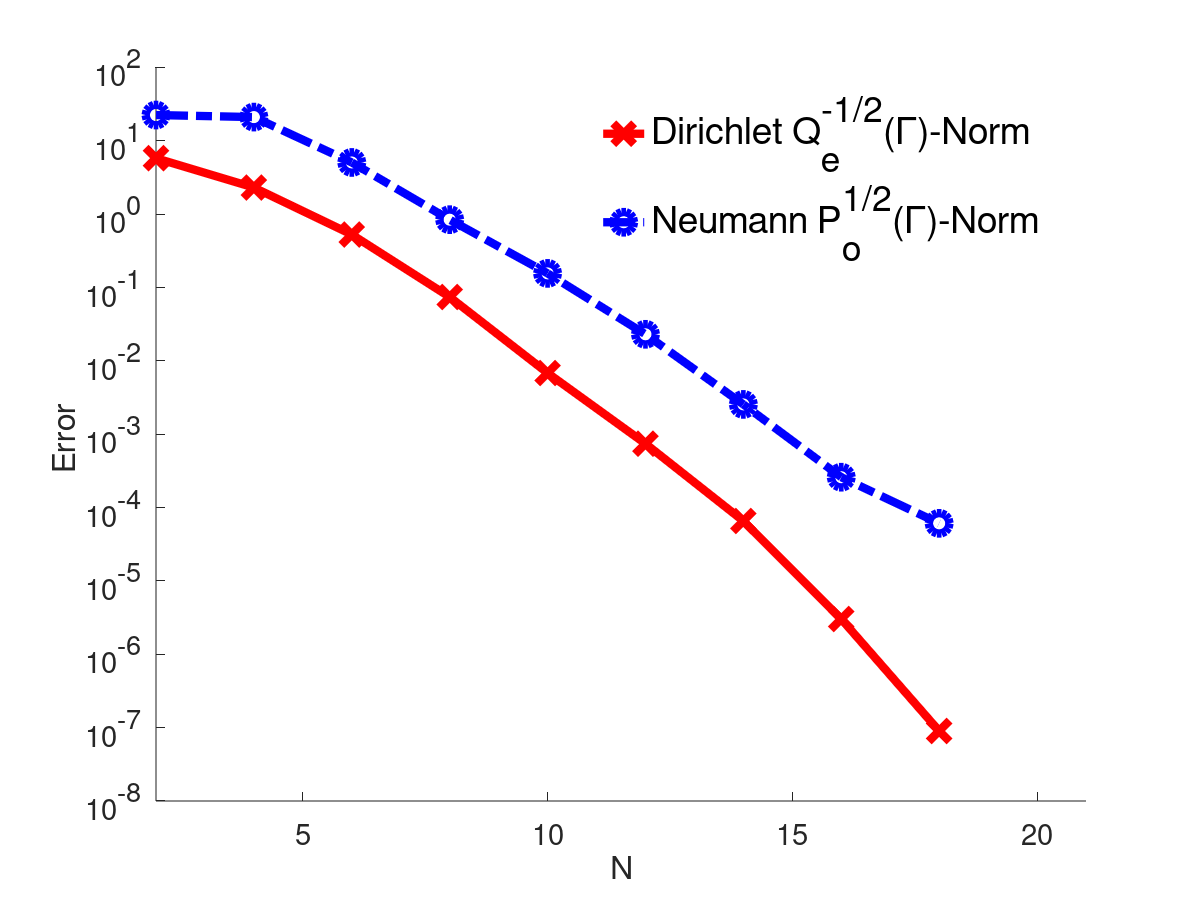}
  }
  \caption{ 
Error in the $Q^{-\half}_e$-Norm for the Dirichlet problems, and $P_o^{\half}$ for the Neumann. Overkill solutions are computed using $N= 20$. Run times for Dirichlet and Neumann problem (with $N=20$) are 749s, 2087s respectively. 
}
\label{fig:Convg1}
\end{figure}

Next, we show the impact of critical points through the following screens. $\Gamma_3$, described by $$\vr_3(r,\theta) = r(1-0.2r^3\cos(3\theta))(\cos\theta,\sin\theta,0),$$
and $\Gamma_4$, with parametrization $$\vr_4(r,\theta) = r(1-0.3r^3\cos(3\theta))(\cos\theta,\sin\theta,0).$$

The case $k=2.8$ is depicted in Figure \ref{fig:Convg2}. While these last two screens are similar, one can see that the error convergence for $\Gamma_4$ is worse than for the other cases. This is explained by the Jacobians' behavior: for $\Gamma_3$, one has 
$$|r(1- 0.2r^3 \cos(3\theta))(1-0.8 r^3 \cos\theta)|,$$
which is no-where null, whereas for $\Gamma_4$ 
$$|r(1- 0.3r^3 \cos(3\theta))(1-1.2 r^3 \cos\theta)|,$$ 
is zero on the curve $1-1.2 r^3 \cos\theta =0$.
\begin{figure}
  \subfigure[
	$\Gamma_3$ 
  ]{   
    \includegraphics[width=.47\textwidth]{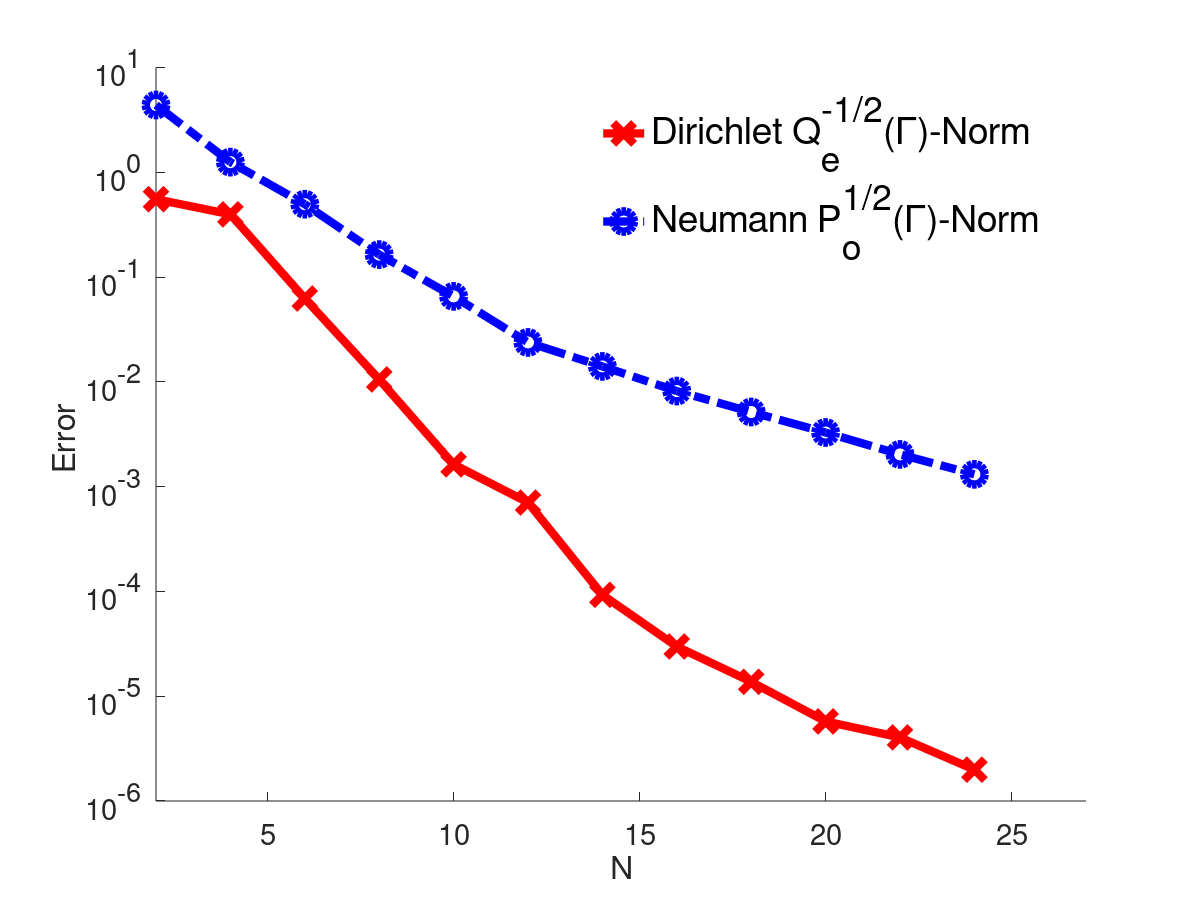}
  }
  \subfigure[
    $\Gamma_4$
  ]{
  \includegraphics[width=.47\textwidth]{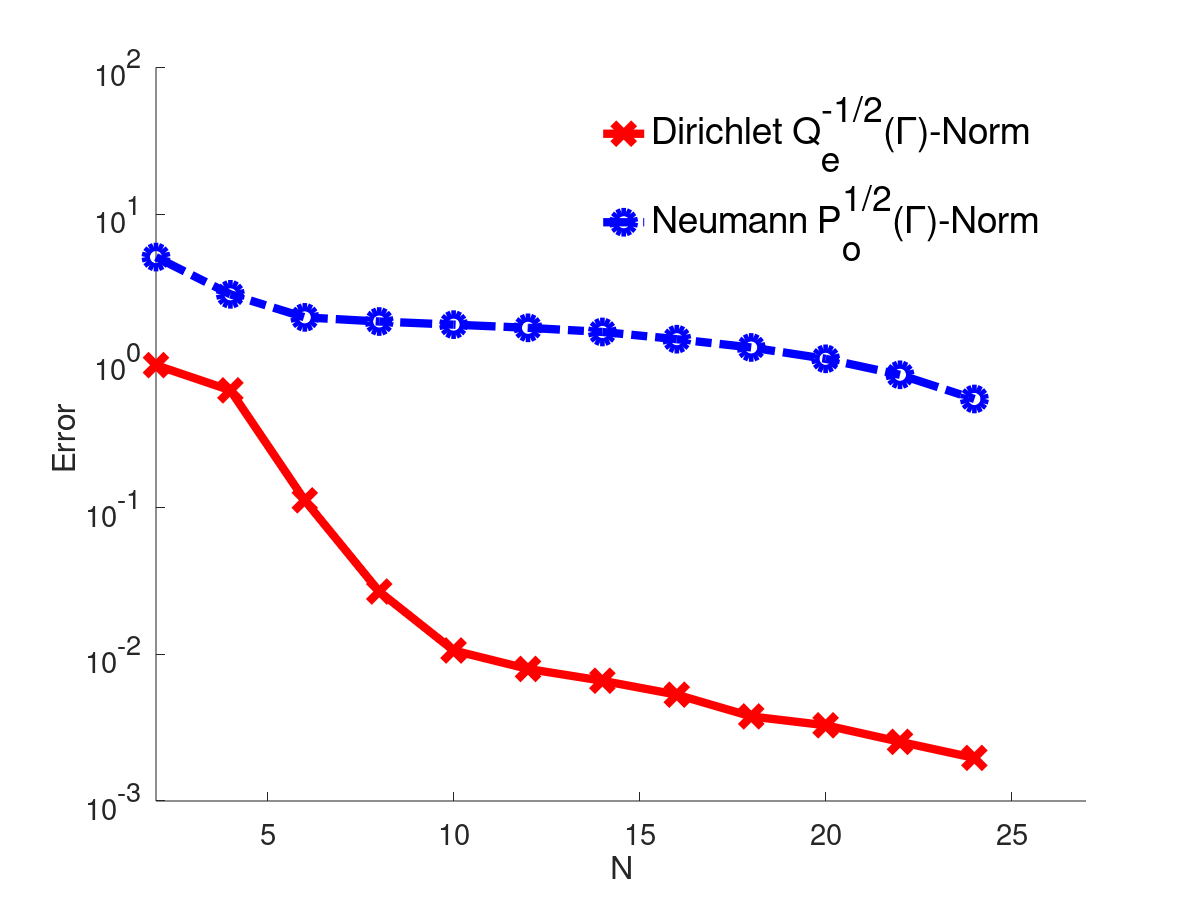}
  }
  \caption{ 
Error in the $Q^{-\half}_e$-Norm for the Dirichlet problems, and $P_o^{\half}$ for the Neumann. Overkill solutions are computed using $N= 28$. Run times for Dirichlet and Neumann problem (with $N=28$) are 2800s, 6198s respectively. 
}
\label{fig:Convg2}
\end{figure}

Finally, and for illustration purposes, we consider a highly complex screen, $\Gamma_5$, along with error convergence and volume solution plots (see Figures \ref{fig:Convg3}(a)--(d)). The screen is given by the parametrization:
%\begin{align*}
%\vr_5(r,\theta) = (r(1-0.2r^3\cos(3\theta))\cos\theta,r(1-0.2r^3\cos(3\theta))\sin\theta,z_5(r,\theta)),
%\end{align*}
\begin{align*}
\vr_5(r,\theta) = (x_5(r,\theta),y_5(r,\theta),z_5(r,\theta)),
\end{align*}
where $(x_5(r,\theta),y_5(r,\theta)) := r(1-0.2r^3\cos(3\theta)) (\cos \theta , \sin \theta )$ and $ z_5(r,\theta) :=
0.2 r \cos(7 y_5(r,\theta) )$.
\begin{figure}
  \subfigure[
	$\Gamma_5$ 
  ]{   
    \includegraphics[width=.47\textwidth]{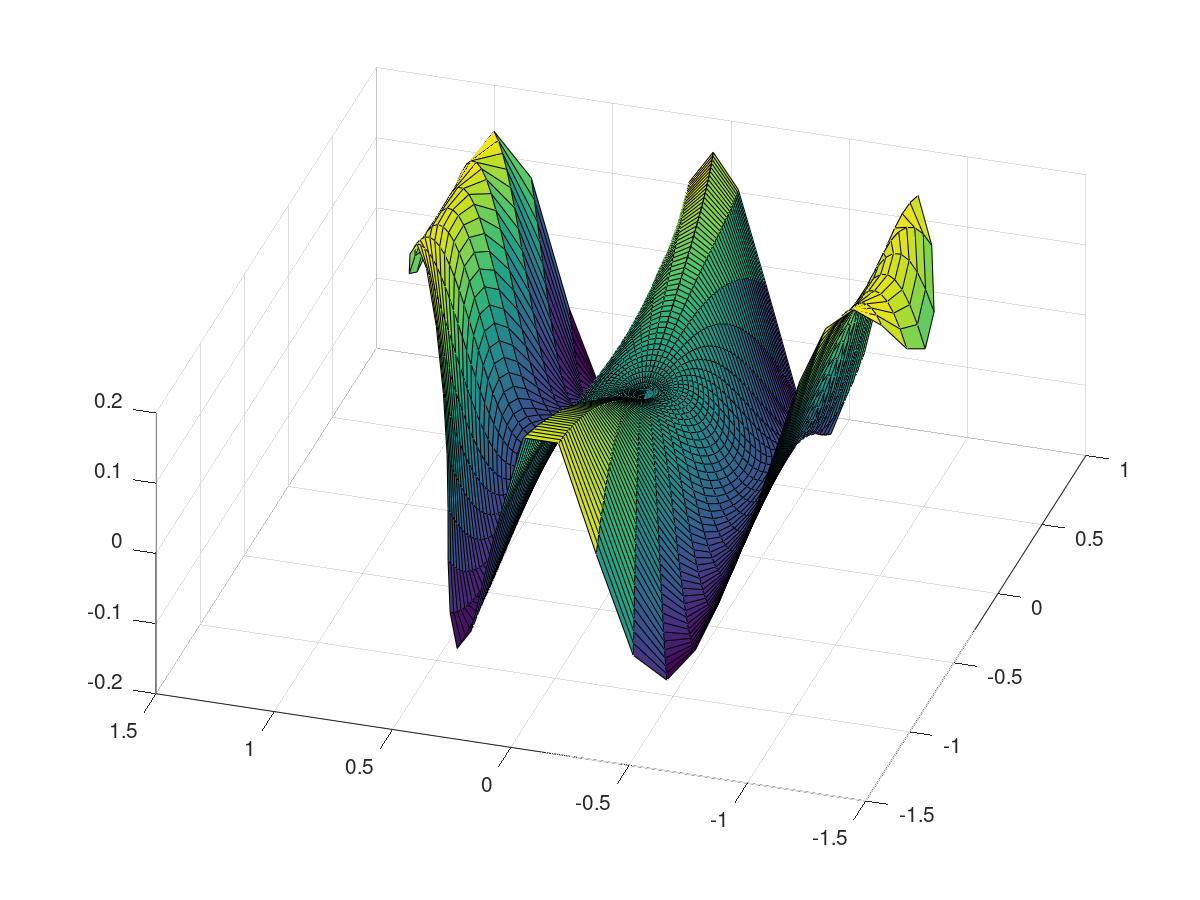}
  }
    \subfigure[
	$\Gamma_5$ 
  ]{   
    \includegraphics[width=.47\textwidth]{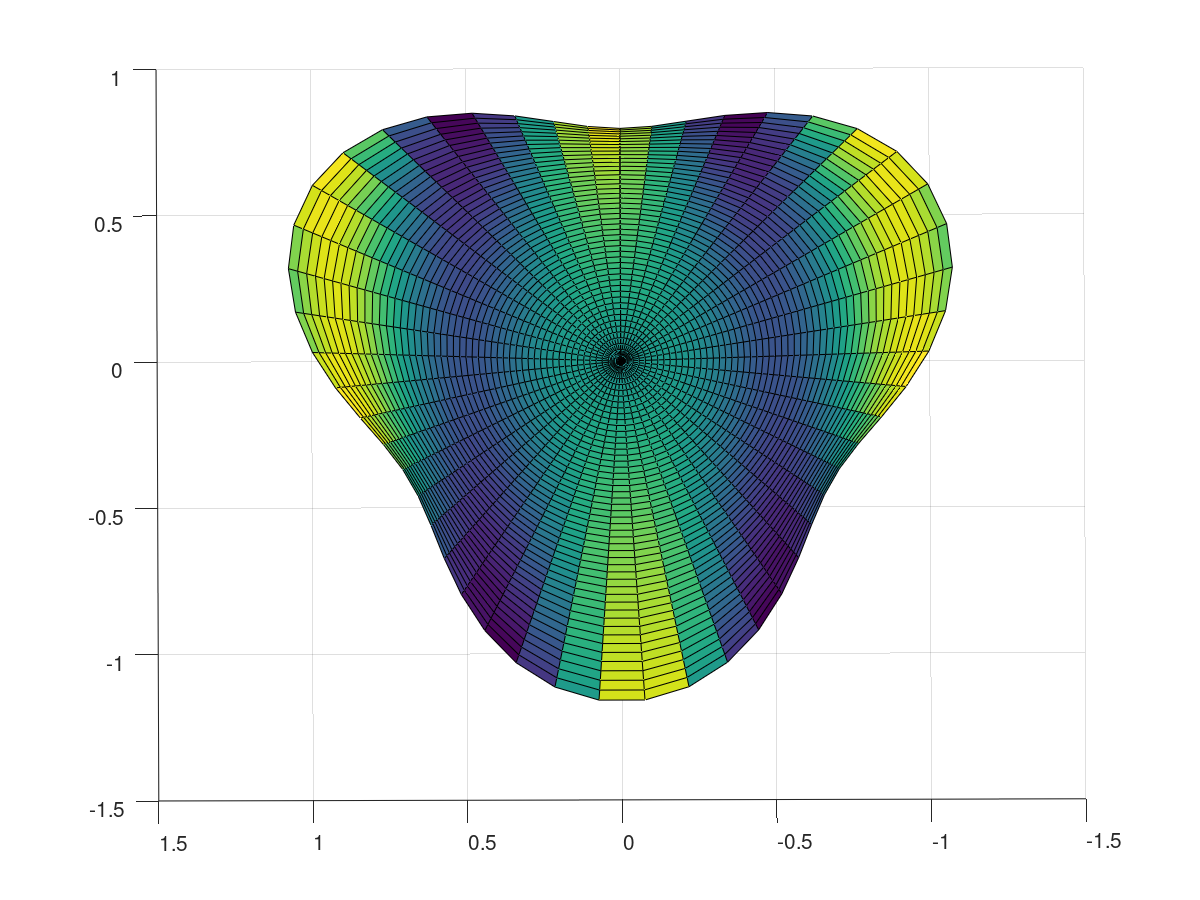}
  }
  \subfigure[
	Convergence 
  ]{   
    \includegraphics[width=.47\textwidth]{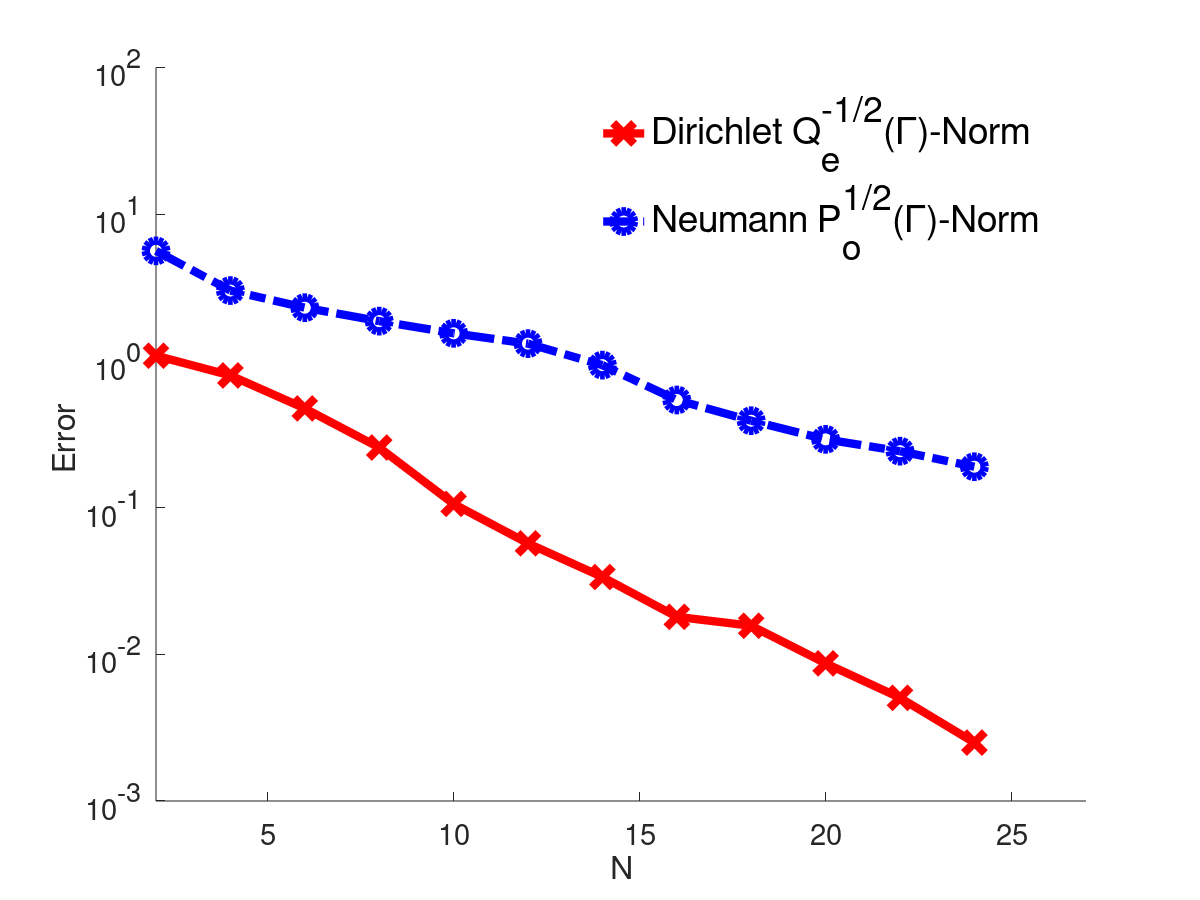}
  }
  \subfigure[
    Volume solution
  ]{
  \includegraphics[width=.47\textwidth]{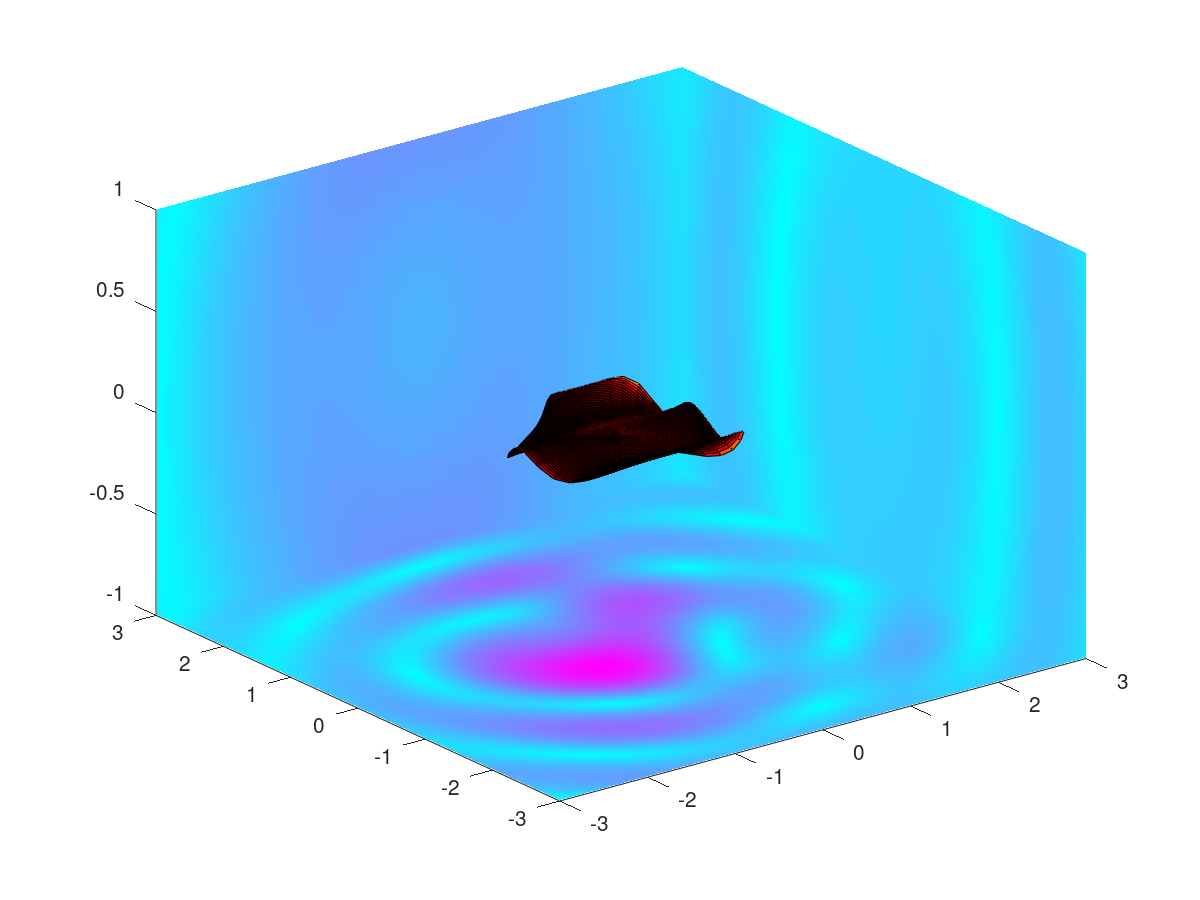}
  }
  \caption{ 
(a)-(b) Screen $\Gamma_5$ profiles. (c) Convergence plots for Dirichlet and Neumann problems computed against an overkill solution computed with $N=28$, for $k= 2.8$, $k_0 = k$, $\theta_0 = \frac{\pi}{2}$, and $\phi_0 = 0$. (d) Plot of the volume solution for the Dirichlet problem.
}
\label{fig:Convg3}
\end{figure}

%%%%%%%%%%%%%%%%%%%%%%%%%%%%%%%%%%%%%%%%%%%%%
\section{Concluding Remarks}
\label{sec:conclusion}
%%%%%%%%%%%%%%%%%%%%%%%%%%%%%%%%%%%%%%%%%%%%%

The present work presents a high-order discretization method for open screens based on weighted projections of the spherical harmonics functions. We have proved that the method converges super-algebraically and also described implementation details. As an efficient solver for the forward problem, future efforts are directed towards using our method to solve shape optimization or inverse problems.

While the analysis presented here describes how approximation errors behave as a function of the number of degrees of freedom, it does not provide any insight on how to choose this number in terms of the problem parameters, in particular, the wavenumber $k$. We refer to \cite{Kexplicit:2012} for a general review on this topic and \cite{cw:2015} for results related to screens. 

%\bibliographystyle{plain}
%\bibliography{biblioFin}

\appendix
\section{Proof of Lemma \ref{lemma:sobRev}}
\label{proof:lemma:sobRev}
%%%%%%%%%%%%%%%%%%%%%%%%%%%%%%%%%%%%%%%%%%%%%
Recall the weakly- and hyper-singular BIOs defined in Section \ref{sec:BIOS}. For $k=0$ and $\Gamma = \mathbb{D}$, we write
\begin{align*}
V_\mathbb{D} u(\vx) &:= \int_\mathbb{D} \frac{1}{4 \pi \| \vx - \vx' \|_2} u (\vx') d \vx', \\ 
W_\mathbb{D} u (\vx) &:= - \gamma_{n,\vx} \int_\mathbb{D} 
\gamma_{n,\vx'} \frac{1}{4 \pi \| \vx-\vx'\|_2} u(\vx') d\vx',
\end{align*}
These operators have two key properties. First, they are continuous and elliptic so they can be used to define equivalent norms. In fact, by \cite[Theorem 3.5.9] {Sauter:2011} we have that
\begin{align}
\label{eq:OpNormsEq}
\| u \|^2_{\widetilde{H}^{-\half}(\mathbb{D})} \cong \langle V_\mathbb{D} u , u \rangle_\mathbb{D},  \quad 
\| u \|^2_{\widetilde{H}^{\half}(\mathbb{D})} \cong \langle W_\mathbb{D} u , u \rangle_\mathbb{D}.
\end{align}
Secondly, we have a characterization of the eigenvalues of these two operators\footnote{See \cite{wolfe1971} for the original proof of the weakly-singular BIO and \cite[Theorem 2.7.1]{RAN17} for the hyper-singular case}:
\begin{align}
\label{eq:OpEigens}
V_\mathbb{D} q^l_m = \frac{1}{4}\Lambda_{l,m} p^l_m, \quad l+m \text{ even},\quad
W_\mathbb{D} p^l_m = \frac{1}{\Lambda_{l,m}} q^l_m, \quad l+m \text{ odd},
\end{align}
where 
\begin{align*}
\Lambda_{l,m} =  \frac{\Gamma \left( \frac{l+|m|+1}{2}\right)\Gamma \left( \frac{l-|m|+1}{2}\right)}{ \Gamma\left( \frac{l+|m|+2}{2}\right)\Gamma \left( \frac{l-|m|+2}{2}\right)},
\end{align*} 
here $\Gamma$ denotes the Gamma function and not an screen. Using Gautschi's inequality \cite{Gautschi59}, it holds that 
\begin{align}
\label{eq:LambdaEst}
\frac{1}{l+1} \lesssim \Lambda_{l,m} \lesssim 
\frac{1}{\sqrt{l+1}}.
\end{align} 

With these elements, we proceed with the proof of Lemma \ref{lemma:sobRev}. Pick any smooth function $u$ on $\mathbb{D}$ and consider its  even lifting to $\mathbb{S}$. Since spherical harmonics are dense on smooth functions defined on the sphere, we can approximate the lifting $u$ by even spherical harmonics. Thus, $u$ can be expanded as $$u= \sum_{l =0}^\infty \sum_{\substack{m=-l\\ m+l \text{ even}}}^l \langle u , q_m^l\rangle p_m^l,$$
now \eqref{lemsobRev1} follows from the orthogonality relation \eqref{eq:diskOrt}.  To prove \eqref{lemsobRev2}, we use the density of functions $q_l^e$ in $\widetilde{H}^{-\half}(\mathbb{D})$, i.e.
$$ u = \sum_{l=0}^\infty \sum_{\substack{m=-l\\ m+l \text{ even}}}^l u^l_m q^l_m,$$
wherein, by orthogonality it holds that $u^l_m = \langle u , p^l_m\rangle$. Hence, computing the norm of $u$ using the equivalence \eqref{eq:OpNormsEq}, the relation  \eqref{eq:OpEigens} and the estimate \eqref{eq:LambdaEst} yields
\begin{align*}
\|u \|_{Q_e^{-\half}(\mathbb{D})} \lesssim \|u\|_{\widetilde{H}^{-\half}(\mathbb{D})} \lesssim \|u \|_{Q_e^{-\fourth}(\mathbb{D})},
\end{align*}
which implies the result. Similar ideas are used to show \eqref{lemsobRev3}. \qed

%%%%%%%%%%%%%%%%%%%%%%%%%%%%%%%%%%%%%%%%%%%%%
\section{Singular Integrals Analysis}
\label{Appendix:quads}
%%%%%%%%%%%%%%%%%%%%%%%%%%%%%%%%%%%%%%%%%%%%%
We consider the integrals $I^c,I^d,I^e$  defined in Section \ref{sec:AproxSingInt}. These integrals have in common that the singularities occur in specifics points in three- or two-dimensional spaces, i.e.~when three or two variables take a specific value. In contrast,  $I^a,I^b$, singularities occur when one variable takes a specific value regardless of the other. 
%%%%%%%%%%%%%%%%%%%%%%%%%%%%%%%%%%%%%%%%%%%%%
\subsection{General idea}
%%%%%%%%%%%%%%%%%%%%%%%%%%%%%%%%%%%%%%%%%%%%%
Let us start with the simpler case of an integral which has a singularity in 2D: 
\begin{align*}
I := \int_0^1 \int_0^1 \frac{1}{\sqrt{x+y}} dy dx,
\end{align*}
the integrand has a singularity at $(x,y) = (0,0)$. Performing the polar change of variables $x= \rho \cos\alpha$, $y= \rho \sin\alpha$, we have that 
\begin{align*}
\begin{split}
I = \int_0^{\frac{\pi}{4}} \int_0^{\frac{1}{\cos\alpha}} \frac{\sqrt{\rho}}{\sqrt{\cos\alpha+\sin\alpha}} d\rho d\alpha +
\int_{\frac{\pi}{4}}^{\frac{\pi}{2}} \int_0^{\frac{1}{\sin\alpha}} \frac{\sqrt{\rho}}{\sqrt{\cos\alpha+\sin\alpha}} d\rho d\alpha.
\end{split}
\end{align*}
Moreover, one can fix the integration domain for the $\rho$ variable by doing a linear change of variable
\begin{align*}
\begin{split}
I = \int_0^{\frac{\pi}{4}} \int_0^{1} \frac{\sqrt{t}}{(\cos\alpha)^{3/2}\sqrt{\cos\alpha+\sin\alpha}} dt d\alpha +\\
\int_{\frac{\pi}{4}}^{\frac{\pi}{2}}  \int_0^{1} \frac{\sqrt{t}}{(\sin \alpha)^{3/2}\sqrt{\cos\alpha+\sin\alpha}} dt d\alpha. 
\end{split}
\end{align*}
These last two integrals can be straightforwardly computed by using a Jacobi rule in $t$ and Gauss-Legendre in $\alpha$, resulting in an optimal convergence rate --exponential in this particular case.  The idea is in fact very simple: use polar coordinates with the origin in the point where the singularity occurs, transferring the multidimensional singularity to the radial coordinate only. The rest of this section gives the detail on each change of variable that is needed and also proving that the resulting integrals are computed with optimal rates. 
%%%%%%%%%%%%%%%%%%%%%%%%%%%%%%%%%%%%%%%%%%%%%
\subsection{Integral $I^c$}
\label{sec:Ic}
%%%%%%%%%%%%%%%%%%%%%%%%%%%%%%%%%%%%%%%%%%%%%
Consider the integral of the form 
\begin{align}
\label{eq:integralJ1}
\begin{split}
J^{c} :=  \int_{\frac{\sqrt{3}}{2}}^1 \int_{\frac{-\pi}{2}}^{\frac{-\pi}{3}}  \int_{0}^{\half}  \frac{ \lambda A^2(r,\beta) }
{ 4 \pi \|\vr(\vx) - \vr(\vx + \lambda A(r,\beta) \ve_{\theta+\beta}) \|}\\
\frac{r d\lambda d\beta dr  }{\sqrt{1-r^2}\sqrt{1- \| \vx +\lambda A(r,\beta)\ve_{\theta+\beta} \|^2}
}, 
\end{split}
\end{align}
where in comparison to \eqref{eq:I1integral}, we restrict to the first interval for the $\beta$ variable as the other cases follows similarly, and we have omitted the integral in the $\theta$ variable as it is not relevant to the singularity analysis and can be treated with Gauss-Legendre quadrature having not effect on the rate of convergence. Since trial and test polynomials are smooth functions they have also been neglected in the ensuing singularity analysis. Consider the following change of variables: $u^2 = 1-r^2$, $\cos\beta = v$. Define $  d:= 4 \pi \|\vr(\vx) - \vr(\vx + \lambda A(r,\beta) \ve_{\theta+\beta}) \|$ and use expansion \eqref{eq:pesoexp} so that \eqref{eq:integralJ1} becomes
\begin{align*}
J^{c}  = \int_0^{\half} \int_{0}^{\half}  \int_{0}^{\half}  \frac{ \lambda d^{-1}
A^2 du dv d\lambda}{\sqrt{1-\lambda}\sqrt{1-v^2} \sqrt{u^2(1+\lambda)-
2\lambda \sqrt{1-u^2} v A}}, 
\end{align*}
wherein by definition of $A(r,\beta)$ (see \eqref{eq:Afun}), we have that 
\begin{align*}
A = \left( \sqrt{u^2+(1-u^2)v^2} -\sqrt{1-u^2} v\right). 
\end{align*}
Apply the first polar change of variables $u = \rho \cos\alpha$, $v = \rho \sin\alpha$ so that
\begin{align*}
A = \rho\left(  \sqrt{1-\rho^2 \cos^2\alpha \sin^2\alpha} 
-\sin\alpha\sqrt{1-\rho^2\cos^2\alpha}\right) =: \rho \widetilde{A}, 
\end{align*}
and we obtain two integrals 
\begin{align}
J^c_{1} :=  \int_0^{\half} \frac{\lambda}{\sqrt{1-\lambda}} 
\int_{0}^{\frac{\pi}{4}} \int_0^{\frac{1}{2 \cos\alpha}}  \frac{(1-\rho^2\sin^2\alpha)^{-1}\rho^2 d^{-1} \widetilde{A}^2 d\rho d\alpha d\lambda}{  \sqrt{\cos^2\alpha(1+\lambda)-2\lambda \sqrt{1-\rho^2 \cos^2\alpha} \sin\alpha\widetilde{A}}}, \\
J^c_{2} :=  \int_0^{\half} \frac{\lambda}{\sqrt{1-\lambda}} 
\int_{\frac{\pi}{4}}^{\frac{\pi}{2}} \int_0^{\frac{1}{2 \sin\alpha}}  \frac{(1-\rho^2\sin^2\alpha)^{-1}\rho^2 d^{-1} \widetilde{A}^2 d\rho d\alpha d\lambda}{  \sqrt{\cos^2\alpha(1+\lambda)-2\lambda \sqrt{1-\rho^2 \cos^2\alpha} \sin\alpha\widetilde{A}}} .
\end{align}
We apply a linear change of variables to fix the integration domain of the $\rho$ variable. For the first integral, it holds that
\begin{align*}
J^c_{1} =  \int_0^{\half} \frac{\lambda}{\sqrt{1-\lambda}} 
\int_{0}^{\frac{\pi}{4}} \int_0^{\half}  \frac{t^2 d^{-1} \widetilde{A}^2 \cos\alpha^{-3} dt d\alpha d\lambda}{ \sqrt{1-t^2\tan^2\alpha} \sqrt{\cos^2\alpha(1+\lambda)-2\lambda \sqrt{1-t^2} \sin\alpha\widetilde{A}}}, 
\end{align*}
wherein $\widetilde{A} = \left(  \sqrt{1-t^2 \sin^2\alpha} -\sin\alpha\sqrt{1-t^2}\right)$. One can see that this integral converges at the optimal rate when we use the Gauss-Legendre rule for all the variables as neither of the terms inside the square roots vanish in the integration domain. For the second integral we have 
\begin{align*}
J^c_{2} =  \int_0^{\half} \frac{\lambda}{\sqrt{1-\lambda}} 
\int_{\frac{\pi}{4}}^{\frac{\pi}{2}} \int_0^{\half}  \frac{t^2 d^{-1} \widetilde{A}^2  \sin\alpha^{-3} dt d\alpha d\lambda}{ \sqrt{1-t^2} \sqrt{\cos^2\alpha(1+\lambda)-2\lambda \sqrt{1-t^2 \cot^2\alpha} \sin\alpha\widetilde{A}}}, 
\end{align*}
and where $\widetilde{A} = \left(  \sqrt{1-t^2 \cos^2\alpha} -\sin\alpha\sqrt{1-t^2\cot^2\alpha}\right)$. In contrast to $J^c_{1}$, one can easily verify that the integrated has one singularity in $(\lambda,\alpha) = (0,\frac{\pi}{2})$, so further transformations are needed. In particular, we use 
$z = \cos\alpha$, $x^2 = \lambda$. Thus, we have that
\begin{align*}
J^c_{2} =  \int_0^{1/\sqrt{2}} \frac{2x^3}{\sqrt{1-x^2}} 
\int_{0}^{1/\sqrt{2}} \int_0^{\half}  \frac{t^2 d^{-1} \widetilde{A}^2  (1-z^2)^{-2} dt dz dx}{ \sqrt{1-t^2} \sqrt{z^2(1+x^2)-2x^2 \sqrt{1-z^2-t^2z^2} \widetilde{A}}}, 
\end{align*}
with $\widetilde{A} = \left(  \sqrt{1-t^2 z^2} -
\sqrt{1-z^2-t^2 z^2}\right).$ Once again we make a polar change of variables $x = \sigma \cos (\phi),$ $z = \sigma \sin (\phi)$,  and we obtain two integrals 
\begin{align*}
\begin{split}
J^c_{2,1} &:=  \int_0^{\half}
\int_{0}^{\frac{\pi}{4}} \int_0^{1/(\sqrt{2}\cos\phi)}    \frac{2 (\sigma \cos\phi)^3 t^2 d^{-1} \widetilde{A}^2  (1-(\sigma \sin\phi)^2)^{-2}} 
{\sqrt{1-(\sigma \cos\phi)^2}
 \sqrt{1-t^2} 
}\\
&\times\frac{d\sigma d\phi dt}{ \sqrt{ \sin^2\phi(1+(\sigma \cos\phi)^2)-2\cos^2\phi \sqrt{1-(\sigma \sin\phi)^2(1+t^2)} \widetilde{A}}}
 , \\
\end{split}
\end{align*}
\begin{align*}
\begin{split}
J^c_{2,2}&:=  \int_0^{\half}
\int_{\frac{\pi}{4}}^{\frac{\pi}{2}} \int_0^{1/(\sqrt{2}\sin\phi)}    \frac{2 (\sigma \cos\phi)^3 t^2 d^{-1} \widetilde{A}^2  (1-(\sigma \sin\phi)^2)^{-2}} 
{\sqrt{1-(\sigma \cos\phi)^2}
 \sqrt{1-t^2} 
}\\
&\times\frac{d\sigma d\phi dt}{ \sqrt{ \sin^2\phi(1+(\sigma \cos\phi)^2)-2\cos^2\phi \sqrt{1-(\sigma \sin\phi)^2(1+t^2)} \widetilde{A}}}
 , \\
\end{split}
\end{align*}
where $\widetilde{A} = \left(  \sqrt{1-t^2 (\sigma \sin\phi)^2} -
\sqrt{1-(\sigma \sin\phi)^2(1+t^2)}\right)$. Finally, we take the corresponding change of variables needed to fix the integration domain for the $\sigma$ variable, and we obtain 
\begin{align*}
\begin{split}
J^c_{2,1} &=  \int_0^{\half}
\int_{0}^{\frac{\pi}{4}} \int_0^{1/\sqrt{2}}    \frac{2 s^3 t^2 d^{-1} \widetilde{A}^2  (1-(s \tan\phi)^2)^{-2}} 
{\cos\phi\sqrt{1-s^2}
 \sqrt{1-t^2} 
}\\
&\times\frac{ds d\phi dt}{ \sqrt{ \sin^2\phi(1+s^2)-2\cos^2\phi \sqrt{1-(s \tan^2\phi)(1+t^2)} \widetilde{A}}}
 , \\
\end{split}
\end{align*}
with $\widetilde{A} = \left(  \sqrt{1-t^2 (s \tan\phi))^2} -\sqrt{1-(s \tan\phi)^2(1+t^2)}\right).$
It is easy to see that the integrand has not singularities and as so it can be integrated with optimum rate using a tenzorisation of the Gauss-Legendre rule. For the second integral we have 
\begin{align*}
\begin{split}
J^c_{2,2} &=  \int_0^{\half}
\int_{\frac{\pi}{4}}^{\frac{\pi}{2}} \int_0^{1/\sqrt{2}}    \frac{2 (s \cot\phi)^3 t^2 d^{-1} \widetilde{A}^2  (1-s^2)^{-2}} 
{\sin\phi\sqrt{1-(s \cot\phi)^2}
 \sqrt{1-t^2} 
}\\
&\times\frac{ds d\phi dt}{ \sqrt{ \sin^2\phi(1+(s \cot\phi)^2)-2\cos^2\phi \sqrt{1-s^2(1+t^2)} \widetilde{A}}}
 , \\
\end{split}
\end{align*}
where $\widetilde{A} = \left(  \sqrt{1-t^2 s^2} -
\sqrt{1-s^2(1+t^2)}\right)$. Once again the integrand does not has singularities and the optimum rate of convergence is retrieved. 

\subsection{Integral $I^d$}
As in the previous case we simplify the Integral $I^d$ in \eqref{eq:I2integral} to the following integral 
\begin{align}
\label{eq:integralJ2}
J^{d} :=  \int_{\frac{\sqrt{3}}{2}}^1 \int_{\frac{-\pi}{3}}^{\frac{\pi}{3}}  \int_{0}^{\half}  \frac{ \lambda A^2(r,\beta) }
{ 4 \pi \|\vr(\vx) - \vr(\vx + \lambda A(r,\beta) \ve_{\theta+\beta}) \|}
\frac{(1-r^2)^{-\half}r d\lambda d\beta dr  }{\sqrt{1- \| \vx +\lambda A(r,\beta)\ve_{\theta+\beta} \|^2}
}, 
\end{align}
where we have only take the first interval for the $\beta$ variable as the other case is similar. We have to make a transformation in $\lambda,r$: start by fixing the singularity $(\lambda,r) = (0,1)$ to the origin of the new variables $r^2 = 1- u^2$, $v^2 = \lambda$, by doing so we obtain 
\begin{align*}
J^{d} :=  \int_{0}^{\half} \int_{\frac{-\pi}{3}}^{\frac{\pi}{3}}  \int_{0}^{1/\sqrt{2}}  \frac{ 2 v^3 d^{-1} A^2
dv d\beta dr  }{\sqrt{1-v^2}\sqrt{(u^2(1+v^2)
-2v^2\sqrt{1-u^2}\cos\beta A}}, 
\end{align*}
where $A = \sqrt{u^2 +(1-u^2)\cos^2\beta}-\sqrt{1-u^2}\cos\beta$. Now, we make the polar change of variables $u = \rho \cos\alpha$, $v = \rho \sin\alpha$, which leads to
\begin{align*}
\begin{split}
J^d_{1} &:= \int_{\frac{-\pi}{3}}^{\frac{\pi}{3}}  \int_{0}^{\arctan(\sqrt{2})}  \int_{0}^{1/(2\cos\alpha)} 
 \frac{ d^{-1} 2 (\rho \sin\alpha)^3 A^2
  }{\sqrt{1-(\rho \sin\alpha)^2}}\\
  &\times \frac{ d\rho d\alpha d\beta }
  {\sqrt{\cos^2\alpha(1+(\rho \sin\alpha)^2)
-2\sin^2\alpha\sqrt{1-(\rho \cos\alpha)^2}\cos\beta A}},
\end{split}
\end{align*}
\begin{align*}
\begin{split}
 J^d_{2} &:= \int_{\frac{-\pi}{3}}^{\frac{\pi}{3}}  \int_{\arctan(\sqrt{2})}^{\frac{\pi}{2}}  \int_{0}^{1/(\sqrt{2}\sin\alpha)} 
 \frac{ d^{-1} 2 (\rho \sin\alpha)^3 A^2
  }{\sqrt{1-(\rho \sin\alpha)^2}}\\
 &\times \frac{ d\rho d\alpha d\beta }
  {\sqrt{\cos^2\alpha(1+(\rho \sin\alpha)^2)
-2\sin^2\alpha\sqrt{1-(\rho \cos\alpha)^2}\cos\beta A}},
\end{split}
\end{align*}
with $A = \sqrt{ (\rho \cos\alpha)^2+(1-(\rho \cos\alpha)^2)\cos^2\beta}-\sqrt{1-(\rho \cos\alpha)^2}\cos\beta.$

Finally we apply the corresponding linear transformation to fix the $\rho$ integration domain, 
\begin{align*}
\begin{split}
J^d_{1} &:= \int_{\frac{-\pi}{3}}^{\frac{\pi}{3}}  \int_{0}^{\arctan(\sqrt{2})}  \int_{0}^{\half}  \frac{ d^{-1} 2 (t \tan\alpha)^3 A^2
   }{\cos\alpha\sqrt{1-(t \tan\alpha)^2}}\\
   &\times \frac{dt d\alpha d\beta }{\sqrt{\cos^2\alpha(1+(t \tan\alpha)^2)
-2\sin^2\alpha\sqrt{1-t^2}\cos\beta A}},
\end{split}
\end{align*}
where $A = \sqrt{ t^2+(1-t^2)\cos^2\beta}-\sqrt{1-t^2}\cos\beta$. This integrand is smooth: only possible singularities can occur when $\rho =0$ but are eliminated by the numerator, and so the rate of convergence is optimal. Similarly for the second part as 
 \begin{align*}
\begin{split}
 J^d_{2}   &:= \int_{\frac{-\pi}{3}}^{\frac{\pi}{3}}  \int_{\arctan(\sqrt{2})}^{\frac{\pi}{2}}  \int_{0}^{1/\sqrt{2}} 
 \frac{ d^{-1} 2 t^3 A^2
  }{\sin\alpha\sqrt{1-t^2}}\\
  &\times\frac{ d\rho d\alpha d\beta }
  {\sqrt{\cos^2\alpha(1+t^2)
-2\sin^2\alpha\sqrt{1-(t \cot\alpha)^2}\cos\beta A}},
\end{split}
\end{align*}
where the new change of variables leads to
\begin{align*}
A = \sqrt{ (t \cot\alpha)^2+(1-(t \cot\alpha)^2)\cos^2\beta}-\sqrt{1-(t \cot\alpha)^2}\cos\beta.
\end{align*}
Once again the integrand is smooth and the optimal convergence rate is retrieved. 

\subsection{Integral $I^e$ }
Finally, we consider the simplification of $I^e$ defined as in \eqref{eq:I3integral}:
\begin{align*}
J^e :=  \int_{\frac{\sqrt{3}}{2}}^1 \int_{\frac{-\pi}{2}}^{\frac{-\pi}{3}}  \int_{\half}^{1}  \frac{ \lambda A^2(r,\beta) }
{ 4 \pi \|\vr(\vx) - \vr(\vx + \lambda A(r,\beta) \ve_{\theta+\beta}) \|}
\frac{r d\lambda d\beta dr  }{\sqrt{1-r^2}\sqrt{1- \| \vx +\lambda A(r,\beta)\ve_{\theta+\beta} \|^2}
}, 
\end{align*}
we have again fixed the value of $\beta$ to the first interval as other cases are similar. The analysis here follows that of the first part of $J_1$ (up until the first polar change of variables). We obtain two integrals, the first one being
\begin{align*}
J^e_{1} =  \int_{\half}^1 \frac{\lambda}{\sqrt{1-\lambda}} 
\int_{0}^{\frac{\pi}{4}} \int_0^{\half}  \frac{t^2 d^{-1} \widetilde{A}^2 \cos^{-3}\alpha dt d\alpha d\lambda}{ \sqrt{1-t^2\tan^2\alpha} \sqrt{\cos^2\alpha(1+\lambda)-2\lambda \sqrt{1-t^2} \sin\alpha\widetilde{A}}}, 
\end{align*}
with $\widetilde{A} = \left(  \sqrt{1-t^2 \sin^2\alpha} 
-\sin\alpha\sqrt{1-t^2}\right)$, and the second one 
\begin{align*}
J^e_{2} =  \int_{\half}^1 \frac{\lambda}{\sqrt{1-\lambda}} 
\int_{\frac{\pi}{4}}^{\frac{\pi}{2}} \int_0^{\half}  \frac{t^2 d^{-1} \widetilde{A}^2  \sin\alpha^{-3} dt d\alpha d\lambda}{ \sqrt{1-t^2} \sqrt{\cos^2\alpha(1+\lambda)-2\lambda \sqrt{1-t^2 \cot^2\alpha} \sin\alpha\widetilde{A}}}, 
\end{align*}
where $\widetilde{A} = \left(  \sqrt{1-t^2 \cos^2\alpha} -\sin\alpha\sqrt{1-t^2\cot^2\alpha}\right)$. Contrary to the first case, these integrands are not smooth due to the term $(1-\lambda)^{-1}$. However, by using a Jacobi rule in $\lambda$ one recovers the optimal convergence rate.

\end{document}